\documentclass[10pt]{article}

\usepackage{amsmath}
\usepackage{amssymb}
\usepackage{amsthm}
\usepackage{amsfonts}
\usepackage{amscd}
\usepackage{graphicx}       
\usepackage[T1]{fontenc}    
\usepackage{hyperref}
\usepackage{array}          
\usepackage{enumitem}       
\usepackage{titling}        
\usepackage[numbers,sort&compress]{natbib}
\usepackage[font=it, labelfont={bf}]{caption}
\usepackage{tikz}
\usepackage[medium]{titlesec}
\usepackage{etoolbox}       
\usepackage{geometry}
\usepackage{todonotes}
\usepackage{algorithm}

\theoremstyle{plain}
\newtheorem{theorem}{Theorem}
\newtheorem{lemma}{Lemma}
\newtheorem{proposition}{Proposition}

\theoremstyle{definition}
\newtheorem{definition}{Definition}
\newtheorem{remark}{Remark}
\newtheorem{example}{Example}

\hypersetup{colorlinks,urlcolor=black!50!green,citecolor=black!45!green,linkcolor=blue}
\makeatletter\patchcmd{\ttlh@hang}{\parindent\z@}{\parindent\z@\leavevmode}{}{}\patchcmd{\ttlh@hang}{\noindent}{}{}{}\makeatother 
\titlespacing*{\section}{0pt}{1mm}{1mm}
\titlespacing*{\subsection}{0pt}{1mm}{1mm}
\titlespacing*{\paragraph}{0pt}{1mm}{1mm}

\newenvironment{Mlist}{\begin{itemize}[topsep=0pt,itemsep=0pt,leftmargin=5mm]}{\end{itemize}}

\newenvironment{Mclaims}{\begin{itemize}[topsep=0pt,itemsep=0pt,leftmargin=7mm]}{\end{itemize}}

\newcommand{\EQN}[1]{(\ref{eqn:#1})}
\newcommand{\LEM}[1]{Lemma~\ref{lem:#1}}
\newcommand{\THM}[1]{Theorem~\ref{thm:#1}}
\newcommand{\SEC}[1]{Section~\ref{sec:#1}}
\newcommand{\PRP}[1]{Proposition~\ref{prp:#1}}
\newcommand{\DEF}[1]{Definition~\ref{def:#1}}

\newcommand{\FIG}[1]{Figure~\ref{fig:#1}}
\newcommand{\RMK}[1]{Remark~\ref{rmk:#1}}
\newcommand{\ALG}[1]{Algorithm~\ref{alg:#1}}
\newcommand{\EXM}[1]{Example~\ref{exm:#1}}

\newcommand{\END}{\hfill \ensuremath{\vartriangleleft}}

\newcommand{\df}[1]{{\it #1}}
\newcommand{\Wlog}{without loss of generality }
\newcommand{\resp}{respectively}
\newcommand{\st}{such that }

\renewcommand{\c}{\colon}

\newcommand{\dto}{\dasharrow}
\newcommand{\bas}[1]{\langle #1\rangle}
\newcommand{\set}[2]{\{#1~|~#2\}}
\newcommand{\aut}{\operatorname{Aut}}
\newcommand{\rnk}{\operatorname{rank}}
\newcommand{\im}{\operatorname{\textbf{1}}}
\newcommand{\zm}{\operatorname{\textbf{0}}}
\newcommand{\bir}{\operatorname{bir}}
\newcommand{\dom}{\operatorname{dom}}
\newcommand{\img}{\operatorname{img}}
\newcommand{\edim}{\operatorname{dim}}
\newcommand{\bmd}{\operatorname{bmd}}
\newcommand{\cdeg}{\operatorname{cdeg}}

\renewcommand{\P}{\mathbb{P}}
\newcommand{\Z}{\mathbb{Z}}

\newcommand{\F}{\mathbb{F}}
\newcommand{\R}{\mathbb{R}}

\renewcommand{\S}{\mathbb{S}}
\renewcommand{\H}{\mathbb{H}}
\newcommand{\E}{\mathbb{E}}
\newcommand{\Q}{\mathbb{Q}}

\newcommand{\bc}{{\operatorname{\textbf{c}}}}
\newcommand{\br}{{\operatorname{\textbf{r}}}}
\newcommand{\bp}{{\operatorname{\textbf{p}}}}
\newcommand{\bq}{{\operatorname{\textbf{q}}}}
\newcommand{\bs}{{\operatorname{\textbf{s}}}}
\newcommand{\bd}{{\operatorname{\textbf{d}}}}

\newcommand{\cP}{\mathcal{P}}
\newcommand{\cR}{\mathcal{R}}
\newcommand{\cS}{\mathcal{S}}
\newcommand{\cI}{\mathcal{I}}
\newcommand{\cJ}{\mathcal{J}}
\newcommand{\cM}{\mathcal{M}}

\newcommand{\bF}{\operatorname{\bf F}}

\newcommand{\fI}{\mathfrak{i}}
\newcommand{\fJ}{\mathfrak{j}}

\newcommand{\fa}{\mathfrak{a}}
\newcommand{\fb}{\mathfrak{b}}
\newcommand{\fc}{\mathfrak{c}}

\newcommand{\fh}{\mathfrak{h}}

\newcommand{\hf}{\hat{f}}
\newcommand{\hg}{\hat{g}}
\newcommand{\hX}{\hat{X}}
\newcommand{\hY}{\hat{Y}}

\newcommand{\e}{e}
\renewcommand{\l}{\ell}
\newcommand{\p}{\varepsilon}
\renewcommand{\k}{\kappa}

\newcommand{\kf}{\kappa_f}
\newcommand{\kg}{\kappa_g}

\newcommand{\khf}{\kappa_{\hf}}
\newcommand{\khg}{\kappa_{\hg}}

\newcommand{\PP}{\P^1\times\P^1}
\newcommand{\mc}[1]{\lfloor#1\rfloor}

\title{Projective isomorphisms between rational surfaces}
\author{Bert J\"uttler, Niels Lubbes, Josef Schicho}

\begin{document}

\setlength{\abovedisplayskip}{3pt}\setlength{\belowdisplayskip}{3pt} 

\maketitle

\begin{abstract}
We present a method for computing projective isomorphisms between
rational surfaces that are given in terms of their parametrizations.
The main idea is to reduce the computation of such projective isomorphisms
to five base cases by modifying the parametric maps such that
the components of the resulting maps have lower degree.
Our method can be used to compute
affine, Euclidean and M\"obius isomorphisms
between surfaces.
\\[2mm]
{\bf MSC1010:} 14J50, 14J26



\end{abstract}

\begingroup
\def\addvspace#1{\vspace{-2mm}}
\tableofcontents
\endgroup

\section{Introduction}
\label{sec:intro}

Suppose we are given two rational surfaces in terms of their parametrizations.
We reduce the computation of the projective isomorphisms
between these surfaces,
to the more tractable problem of finding projective isomorphisms
between surfaces that are covered by lines or conics
and that belong to one of five different types.
We shall discuss two of the types in more detail and
leave the remaining three types as future work.
Our reduction method translates into \ALG{alg}
and its correctness follows from \THM{C} and \THM{B}.
We refer to \citep[\tt github]{github} for a (partial) implementation.
We explain in \SEC{A} how to recover
from projective isomorphisms between surfaces,
the affine, Euclidean and M\"obius isomorphisms.


If $f$ is a rational map, then we denote its \df{domain} by $\dom f$
and the \df{Zariski closure of its image} by $\img f$.
Let $\cM$ be defined as the set of all rational maps
$f\c\dom f\dto\img f\subseteq\P^{\edim f}$
defined over some algebraically closed field $\F$
\st $\dom f\in\{\P^2,\PP\}$.
Here $\edim f\in\Z_{\geq 0}$ denotes the \df{embedding dimension}
and we assume that $\img f$ is not contained in a
hyperplane section of~$\P^{\edim f}$.
The group of \df{biregular automorphisms} of a variety $Z$ is denoted by $\aut(Z)$
and it is known that $\aut(\P^{\edim f})$ is linear \citep[Example~7.1.1]{har}.
If $f,g\in\cM$, then
the \df{set of projective isomorphisms} is defined as
\[
\cP(f,g):=\set{p\in\aut(\P^{\edim f})}{p(\img f)=\img g}.
\]
In this article we address the following problem:
\begin{center}
\it For given birational maps $f,g\in\cM$ determine $\cP(f,g)$.
\end{center}
A projective isomorphism $p\c \img f\to \img g$ induces a birational map
\[
g^{-1}\circ p\circ f\c \dom f\dto \dom g.
\]
Our strategy is to compute the set of birational
maps between the domains that are compatible with projective isomorphisms
in this way.
In \cite{hau} it was assumed that $\dom f=\dom g=\P^2$ with both $f$ and $g$ base point free
so that $g^{-1}\circ p\circ f$ is a linear projective automorphism of $\P^2$.
In this article, we admit arbitrary base points and
as a consequence we get compatible birational maps
that blow up base points and contract curves.
Fortunately, these blowups and contractions can be controlled
using adjunction theory, and allow us to get a finite-dimensional set
of candidates for the induced birational maps between the domains.
Once we have this set of candidates, we write down the
conditions expressing the statement that we really have a projective
isomorphism. This will be a system of algebraic equations and its
solutions correspond to the projective isomorphisms in $\cP(f,g)$.

Isomorphisms between surfaces are of interest in
geometric modeling.
If $\img f$ and $\img g$ are surfaces in $\P^3$ that are covered by lines,
then \citep[Section~3.1]{alc2} and \citep[Algorithm~5]{vrsek} provide methods for
computing affine and Euclidean isomorphisms between~$\img f$ and~$\img g$.
We refer to \citep[Introduction]{alc1} and \citep[Introduction]{hau} for further references.

If $\img f$ is smooth, then $\aut(\img f)$
and its action on the N\'eron-Severi lattice~$N(\img f)$
is of interest to algebraic geometers \cite{dol,zhang,koita}.
The relation of $\aut(\img f)$ to our paper is as follows:
projective isomorphisms in $\cP(f,f)$ correspond to automorphisms of $\img f$ whose induced action on $N(\img f)$ leaves
the class of hyperplane sections invariant.
We will clarify these notions in \SEC{pre} in order to make
our results also accessible to the geometric modeling community.
We refer to \citep[Introduction]{dol} for further references
from the viewpoint of algebraic geometry.

The \df{set of compatible reparametrizations} is defined as
\[
\cR(f,g):=\set{r \in \bir(\dom f,\dom g)}{ p\circ f= g\circ r \text{ for some } p\in\cP(f,g) },
\]
where $\bir(\dom f,\dom g)$ is the set of birational maps between the domains.
The idea of \ALG{alg} is to first compute a set $\cS$ that contains the compatible reparametrizations $\cR(f,g)$.
We explain in \SEC{P} how to recover the projective isomorphisms $\cP(f,g)$ from this super set $\cS$.
In \SEC{cr} we state \THM{C} and \THM{B} which reduce the computation of
$\cS$ to five base cases B1---B5.
The base cases B1 and B2 are considered in \SEC{B},
and the remaining three base cases are left as future work.
In \SEC{A} we discuss some applications of our algorithm.
See \EXM{B2} for a full run of the algorithm in a concrete instance.
Finally, we present the proof for \THM{C} and \THM{B} in \SEC{proofs}.

\section{Basic concepts and notation}
\label{sec:pre}

In order to make this article accessible to a wide audience we recall
some basic concepts from algebraic geometry and provide references.
We will also introduce non-standard notation that will be used in the remaining sections of this article.

We define a \df{sequential blowup}
as a birational morphism $\pi\c Z_{r+1}\to Z_1$
between smooth surfaces
together with blowups
$\pi_i\c Z_{i+1}\to Z_{i}$ of points~$p_i\in Z_i$ for $1\leq i\leq r$
\st $\pi=\pi_1\circ\ldots\circ\pi_r$.
We refer to $p_i$ as the \df{center of the blowup} and $E_{i+1}:=\pi_i^{-1}(p_i)$
as a \df{$(-1)$-curve} (a curve isomorphic to~$\P^1$ and with self-intersection~$-1$).
See \citep[Example~I.4.9.1 and Section V.3]{har} for more information.
If $p_i\in (\pi_{i-1}\circ\cdots\circ\pi_j)^{-1}(p_j)$
for some $r\leq i<j\leq 1$,
then we say that $p_i$ is \df{infinitely near} to~$p_j$.
We call a point
\df{infinitely near} if it is infinitely near to some point and
\df{simple} otherwise.

Suppose $C_1\subset Z_1$ is a curve and that
$C_2\subset Z_2$ is the Zariski closure
of the preimage~$\pi_1^{-1}(C_1\setminus\{p_1\})$.
We refer to \citep[Remark~V.3.5.2]{har} for the definition of multiplicity
of $C_1$ at the simple point~$p_1$.
The \df{multiplicity} of~$C_1$ at an infinitely near point~$p_2$
is defined as the usual multiplicity of~$C_2$ at~$p_2$.

Let $V$ be a vector space of forms on~$Z_1$.
The \df{linear series} of $V$ is defined as
$|V|:=\set{ \operatorname{ZeroSet}(v)}{v\in V}$.
The \df{moving part} of~$V$
is defined as the vector space that is generated
by the polynomial quotients $g_1/q,\ldots,g_n/q$,
where $\bas{g_1,\ldots,g_n}_\F$ is a basis for $V$
and $q:=\gcd(g_1,\ldots,g_n)$ is the greatest common polynomial divisor.

Suppose that $f\c Z_1\dto \P^n$ is the rational map
whose components generate~$V$.
The \df{associated vector space}~$V_f$ of $f$ is defined as $V$.
We say that $q$ is a \df{base point} of \df{multiplicity} $m$
of both $V$ and $f$, if there exists a sequential blowup $\pi$
\st $q=p_i$ for some $1\leq i\leq r$
and if a general curve in $|V|$ has multiplicity $m>0$ at~$p_i$.

\begin{remark}[\ALG{get} and \ALG{set}]
We refer to \citep[Algorithms~1 and 2]{n-bp}
for the method and implementation of \ALG{get} and \ALG{set}
(see alternatively \cite{rito} for a possibly faster implementation).
We remark that in \cite{n-bp} a sequential blowup $\pi$
and its centers are represented in
terms of a data structure that extracts only the
part of $\pi$ that is needed for this article.
\END
\end{remark}

\begin{algorithm}[!ht]
\caption{}
\label{alg:get}
\begin{itemize}[itemsep=0pt,topsep=5pt,leftmargin=5mm]
\item {\bf input.}
A vector space $V$ of forms on $Z_1\in\{\P^2,\P^1\times\P^1\}$.
\item {\bf output.}
A sequential blowup $\pi\c Z_{r+1}\to Z_1$ and $m_i\in\Z_{>0}$
\st $q$ is a base point of~$V$ of multiplicity $m_i$
if and only if
there exists a unique $1\leq i\leq r$
\st $q=p_i\in Z_i$.
\end{itemize}
\end{algorithm}

\begin{algorithm}[!ht]
\caption{}
\label{alg:set}
\begin{itemize}[itemsep=0pt,topsep=5pt,leftmargin=5mm]
\item {\bf input.}
A sequential blowup $\pi\c Z_{r+1}\to Z_1$ with centers $p_i\in Z_i$
\st $Z_1\in\{\P^2,\P^1\times\P^1\}$.
The vector space~$W$ of all (bi-) degree~$d$ of forms on~$Z_1$.
A set of multiplicities $m_i\in\Z_{>0}$ for $1\leq i\leq r$.

\item {\bf output.}
The subspace~$V\subset W$ of forms whose zero set are curves
that have multiplicity $\geq m_i$ at the base point $p_i$ for all $1\leq i\leq r$.
\end{itemize}
\end{algorithm}

\begin{example}
\label{exm:inf}
Suppose that $f\c \P^2\dto \P^1$ maps
$(x_0:x_1:x_2)$ to $(x_1^2+x_2^2:x_2^2+x_1x_0)$.
We find that $p_1:=(1:1:\fI)$, $p_2:=(1:1:-\fI)$
and $p_3:=(1:0:0)$ are simple base points for $f$ with multiplicities $(m_1,m_2,m_3)=(1,1,1)$.
The map~$f$ has also a base point $p_4$ of multiplicity $m_4=1$ that is
infinitely near to~$p_3$.
\END
\end{example}

\begin{definition}
\label{def:map}
Suppose that $V$ is a vector space of forms on a smooth projective surface~$Z$.
Let $\P(V^*)$ denote the projectivization of the dual space of~$V$
so that each point in~$\P(V^*)$ corresponds to a codimension one subspace of~$V$.
Let $\breve{\varphi}_V\c Z\dto \P(V^*)$ be defined
as $\breve{\varphi}_V(p):=\set{v\in V}{v(p)=0}$ for all $p\in Z$.
A \df{choice of basis} is defined as an isomorphism
$\beta\c \P(V^*)\to\P^{\dim V-1}$.
The \df{associated map} $\varphi_V\c Z\dto \P^{\dim V-1}$
is defined as $\beta\circ\breve{\varphi}_V$ where $\beta$
is a choice of basis.
We need to be careful that the definitions and assertions in this article
that involve the notion of the associated map
are independent of the choice of such a basis.
Recall that we denote by $\img\varphi_V$
the Zariski closure of the image of~$\varphi_V$.
\END
\end{definition}

\begin{definition}
\label{def:M}
In this article we will assume that the components
of maps in~$\cM$ as defined in \SEC{intro}
have a constant greatest common divisor.
The \df{component degree}~$\cdeg(f)$ of $f\in\cM$
is defined as the (bi-) degree of the components of~$f$.
Let the sequential blowup~$\pi\c S\to\dom f$ be the output of \ALG{get}
when it is applied to the associated vector space~$V_f$.
In this case, we call $\bmd f:=S$ the \df{base model} for~$f$.
\END
\end{definition}

\begin{definition}
Suppose that $f\in\cM$ has base points
$p_1,\ldots,p_r$.
In this article the \df{N\'eron-Severi lattice} $N(\bmd f)$ is
an additive group together with
an intersection product $\cdot\c N(\bmd f)\otimes N(\bmd f)\to\Z$
that satisfies the following axioms:
\begin{Mlist}

\item If $\dom f=\P^2$, then
$N(\bmd f)\cong\bas{\e_0,\e_1,\ldots,\e_r}_\Z$,
where the only non-zero intersections between the
generators are
$\e_0^2=1$ and $\e_i^2=-1$ for $1\leq i\leq r$.

\item If $\dom f=\PP$, then
$N(\bmd f)\cong\bas{\l_0,\l_1, \p_1,\ldots,\p_r}_\Z$,
where the only non-zero intersections between the
generators are
$\l_0\cdot\l_1=1$ and $\p_i^2=-1$ for $1\leq i\leq r$.
\end{Mlist}
See forward \RMK{N} for more information.
\END
\end{definition}

\begin{definition}
\label{def:nota}
Suppose that $f\in\cM$ has base points $p_1,\ldots,p_r$
with multiplicities $m_1,\ldots,m_r$, \resp.
First suppose that $\dom f=\P^2$ and that $d:=\cdeg f$.
The \df{class} of~$f$ is defined as
\[
[f]=d\,\e_0-m_1\,\e_1-\ldots-m_r\,\e_r.
\]
The \df{greatest common divisor} of $f$ is defined as
\[
\gcd [f]:=\gcd(d,m_1,\ldots,m_r).
\]
The \df{canonical class} associated to $f$ is defined as
\[
\kf:=-3\,\e_0+\e_1+\ldots+\e_r.
\]
Conversely, suppose that $c:=d\,\e_0-m_1\,\e_1-\ldots-m_r\,\e_r$
is a class in $N(\bmd f)$ \st $d,m_1,\ldots, m_r>0$.
The \df{associated vector space} $V_c$ is defined as
the output of \ALG{set} with input $\pi\c\bmd f\to\dom f$ and $d,m_1,\ldots,m_r$.
We denote
\[
h^0(c):=\dim V_c.
\]
If $M$ is the moving part of $V_c$, then
the \df{parametric map} of~$c$ is
defined as
\[
\Psi_c:=\varphi_M\c\dom f\to \P^{h^0(c)-1}.
\]
The \df{moving part} of the class~$c$
is defined as the following class:
\[
\mc{c}:=[\Psi_c].
\]
If $\dom f=\PP$ and $\cdeg(f)=(d_1,d_2)$, then
the terminology is analogous except:
\begin{gather*}
[f]=d_1\,\l_0+d_2\,\l_1-m_1\,\p_1-\ldots-m_r\,\p_r,
\\
\gcd [f]:=\gcd(d_1,d_2,m_1,\ldots,m_r)
\quad\text{ and }\quad
\kf:=-2\,\l_0-2\,\l_1+\p_1+\ldots+\p_r.
\end{gather*}
The reader is warned that the non-standard notation introduced
in this definition will be used throughout this article.
\END
\end{definition}

\begin{example}
\label{exm:cdeg3}
The birational map $f\c\P^2\dto\P^3$ defined by
\[
x\mapsto (x_1^3 - x_1^2x_0: x_1^2x_2: x_1x_2^2: x_1x_2x_0 + x_2^3 - x_2^2x_0),
\]
has simple base points $p_1:=(1:0:0)$, $p_2:=(1:1:0)$ and $p_3:=(1:0:1)$ with multiplicities
$m_1:=2$, $m_2:=1$ and $m_3:=1$, \resp.
The class of $f$ is
$[f]=3\,\e_0-2\,\e_1-\e_2-\e_3$,
and for a particular choice of a basis,
the parametric map $\Psi_{[f]}\c\P^2\dto \P^4$ is defined as
$
x\mapsto (x_1^3 - x_1^2x_0: x_1^2x_2: x_1x_2^2: x_1x_2x_0: x_2^3 - x_2^2x_0).
$
Notice that $\edim f=3<h^0([f])-1=4$ and that $\bmd f$ is $\P^2$ blown up in $p_1$, $p_2$ and $p_3$.
Since $f$ is birational,
$\deg(\img f)$ is equal to the number of intersections
outside the base points of the pullback along~$f$ of two hyperplane sections of $\img f$
to $\P^2$,
and therefore we can deduce that $\deg(\img f)=[f]^2=3$.
In fact, for all $f\in \cM$,
either $\dim(\img f)<2$ or $\deg(\img f)=\deg(f)\cdot [f]^2$,
where $\deg(f)$ equals the number of points in a general fiber.
\END
\end{example}


\begin{example}
\label{exm:cdeg22}
Suppose that $f\c\PP\dto\P^4$ is defined by
\begin{gather*}
(y_0:y_1;y_2:y_3)
\mapsto
(y_0^2 y_2^2 -3 y_1^2 y_3^2:
y_0^2 y_2 y_3 + 3 y_1^2 y_2 y_3:
y_0^2 y_3^2 + 3 y_1^2 y_3^2:\\
y_0 y_1 y_2^2 + y_0 y_1 y_3^2:
y_1^2 y_2^2 + y_1^2 y_3^2).
\end{gather*}
The components of this map form the vector space $V_{[f]}$ of forms of bi-degree (2,2)
that have four simple base points of multiplicity one ($\fI^2=-1$ and $\fJ^2=-\frac{1}{3}$):
\begin{gather*}
\begin{array}{ll}
p_1:=(1:-\fJ;1:\fI), & p_3:=(1:-\fJ;1:-\fI),\\
p_2:=(1:\fJ;1:-\fI), & p_4:=(1:\fJ;1:\fI).
\end{array}
\end{gather*}
Thus $[f]=2\,\l_0+2\,\l_1-\p_1-\p_2-\p_3-\p_4$
so that $\deg(\img f)=\deg(f)\cdot[f]^2=4$ with $\deg(f)=1$ and $h^0([f])=5$.
Let $\tau_1$ and $\tau_2$ denote the projection $\PP\to\P^1$
to the first and second component, \resp.
Notice that the base points do not lie in general position,
since
$\tau_1(p_1)=\tau_1(p_3)$,
$\tau_1(p_2)=\tau_1(p_4)$,
$\tau_2(p_1)=\tau_2(p_4)$ and
$\tau_2(p_2)=\tau_2(p_3)$.
Therefore, it follows that
\begin{gather*}
\begin{array}{cc}
h^0(\l_0-\p_1-\p_3)=1, & h^0(\l_1-\p_1-\p_4)=1,\\
h^0(\l_0-\p_2-\p_4)=1, & h^0(\l_1-\p_2-\p_3)=1.
\end{array}
\end{gather*}
The fibers of $\tau_1$ and $\tau_2$ that contain two base points are
contracted via $f$ to four complex conjugate isolated singularities in $\img f$.
We remark that $\img f$ can be linearly projected to a quartic surface in $\P^3$ whose
real points form a torus of revolution in $\R^3$.
\END
\end{example}

\begin{remark}
\label{rmk:N}
We shall consider
maps $f\in\cM$ whose base points are contained in some fix set of base points
$\{p_1,\ldots,p_r\}$ and the classes of such maps keep track
of the multiplicities at these base points.
We may think of the class of a map
as a ``generalized component degree'' and the intersection product
between these classes allows us to access power tools from algebraic geometry.
The reason is that the N\'eron Severi lattice~$N(\bmd f)$
is actually the set of divisor classes
on~$\bmd f$ modulo numerical equivalence.
In this article we consider only rational surfaces and
thus numerical equivalence and rational equivalence define the same
equivalence relation on divisor classes.
The element~$\kf$
is the ``canonical class'' of~$\bmd f$ and is
defined as the divisor class that corresponds to the line bundle
that is the determinant of the cotangent bundle
of~$\bmd f$.
We refer to \citep[Section~1.1]{laz1} for more information.
\END
\end{remark}

\section{Isomorphisms from reparametrizations}
\label{sec:P}

Suppose that $f,g\in\cM$ are birational
and suppose that~$\cS\supseteq \cR(f,g)$
is a family $(r_c)_{c\in \cI}$ of reparametrizations indexed by $\cI\subseteq \F^t$
for some $t>0$.
In order to recover the projective isomorphisms $\cP(f,g)$ from $\cS$ we
show how to recover the index-set $\cJ\subseteq \cI$ so that $\cR(f,g)=(r_c)_{c\in \cJ}$.

\begin{definition}
Let $\vec{v}$ be a column vector with $m$ rows that consists of a basis for all forms of the same (bi-) degree.
If $f$ has $n+1$ components, then the \df{coefficient matrix} of~$f$ is
defined as the $(n+1)\times m$ matrix $M_f$ \st $M_f \cdot \vec{v}$ defines $f$ as a column vector.
Let $\ker M_f$ be a matrix whose columns form a basis for the kernel of $M_f$.
We denote the identity matrix and zero matrix by $\im$ and $\zm$, \resp.
\END
\end{definition}

\begin{example}
If $f\in\cM$ \st $\dom f=\PP$, $\img f\subset\P^3$ and $\cdeg(f)=(2,2)$,
then we may choose
\[
\vec{v}=(
s^2u^2,\,
s^2uv, \,
s^2v^2,\,
stu^2, \,
stuv,  \,
stv^2, \,
t^2u^2,\,
t^2uv, \,
t^2v^2)^\top,
\]
and we find that $M_f$ is a $4\times 9$ matrix.
\END
\end{example}

The index set $\cJ$ \st $\cR(f,g)=(r_c)_{c\in \cJ}$ is recovered as follows:
\begin{multline}
\label{eqn:J}
\cJ:=\big\{ c\in \cI ~|~ g\circ r_c \text{ has the same base points as } f,
\\
~\cdeg(g\circ r_c)=\cdeg(f)
\quad \text{and}\quad M_{g\circ r_c}\cdot \ker M_f=\zm
\big\}.
\end{multline}
The composition $g\circ r_c$ is computed using first substitution
after which we factor out the greatest common divisor of the resulting components.
For computing the base points of $f$ and
enforcing these base points on $g\circ r_c$
we use \ALG{get} and \ALG{set}.

In order to recover $\cP(f,g)$ from $\cR(f,g)=(r_c)_{c\in\cJ}$,
we consider the following two $m\times m$ matrices that are
constructed using matrix augmentation~$(\cdot|\cdot)$:
\[
E_f:=\bigl(M_f^\top~|~ \ker M_f \bigr)^\top
\qquad\text{and}\qquad
E_{g\circ r_c}:=\bigl(M_{g\circ r_c}^\top ~|~ \ker M_f \bigr)^\top.
\]
If $U\in \F^{n+1\times m+1}$ is a matrix, then we denote
by $\chi_{_U}\c\P^m\to\P^n$ the corresponding projective linear map.

\begin{proposition}
\label{prp:S}
$\cP(f,g)=\set{ \chi_{_U}\in \aut(\P^n)  }{ U\oplus\im=E_{g\circ r_c} \cdot E_f^{-1}  \text{ and }  c\in \cJ}$.
\end{proposition}

\begin{proof}
Let the components of $v\c\dom f\to Z\subset\P^m$ form a basis of all forms of degree~$\cdeg(f)$.
Notice that $\img f,\img g\subset\P^n$ and $m\geq n$.
Let $s\c Z\dto \img f$ and $t\c Z\dto \img g$
be the linear projections defined by the coefficient matrices $M_f$ and $M_{g\circ r_c}$, \resp.
We require that $M_{g\circ r_c}$ and $M_f$ have the same kernel
so that the projection centers of $s$ and $t$ coincide.
Indeed, the elements in $\cP(f,g)$ are via $s$ and $t$ compatible
with the projective isomorphisms of $Z$ that preserve this center of projection.
In other words, if $\cJ'=\set{c\in\cI}{r_c\in \cR(f,g)}$, then
\begin{gather*}
\cP(f,g)
=
\set{\chi_{_U}\in\aut(\P^n)}{U\cdot M_f=M_{g\circ r_c} \text{ for some }  c\in \cJ'}
\\
=
\set{\chi_{_U}\in\aut(\P^n)}{(U \oplus \im)\cdot  E_f=E_{g\circ r_c} \text{ for some }  c\in \cJ'}.
\end{gather*}
It remains to show that $\cJ'\subseteq \cJ$.
If $c\in\cJ'$, then there exists a projective isomorphism $p\in\cP(f,g)$ \st $p\circ f=g\circ r_c$
and thus $c\in\cJ$ by definition, which concludes the proof.
\end{proof}

We define the following index sets, where
a matrix is \df{normalized} if it has non-zero determinant and
if the first non-zero entry of the matrix has value one:
\begin{gather*}
\cI_{\P^2}:=
\left\{
c\in\F^9
~|~
\left(\begin{smallmatrix}
c_0 & c_1 & c_2 \\
c_3 & c_4 & c_5 \\
c_6 & c_7 & c_8 \\
\end{smallmatrix}\right)
\text{ is normalized}
\right\},
\\
\cI_{\PP}:=
\left\{
c\in\F^8
~|~
\left(
\begin{smallmatrix}
c_0 & c_1 \\
c_2 & c_3 \\
\end{smallmatrix}
\right)
\text{ and }
\left(
\begin{smallmatrix}
c_4 & c_5 \\
c_6 & c_7 \\
\end{smallmatrix}
\right)
\text{ are both normalized}
\right\}.
\end{gather*}
We assume coordinates $x=(x_0:x_1:x_2)$ and $y=(y_0:y_1;y_2:y_3)$ for $\P^2$ and $\P^1\times\P^1$, \resp.
If $c\in \cI_{\P^2}$, then we implicitly assume that the corresponding reparametrization $r_c\colon\P^2\to\P^2$ is defined as
\[
r_c\c x\mapsto
(
c_0\,x_0 + c_1\,x_1 +  c_2\,x_2:
c_3\,x_0 + c_4\,x_1 +  c_5\,x_2:
c_6\,x_0 + c_7\,x_1 +  c_8\,x_2
).
\]
Similarly, if $c\in \cI_{\PP}$, then $r_c\colon\PP\to\PP$ is defined as
\[
r_c\c y\mapsto
(
c_0\,y_0 + c_1\,y_1:
c_2\,y_0 + c_3\,y_1
;
c_4\,y_2 + c_5\,y_3:
c_6\,y_2 + c_7\,y_3
).
\]
We denote by $\aut_\circ(\PP)$ the identity component of $\aut(\PP)$.

\begin{lemma}
\label{lem:aut}
If the set of reparametrizations $\cS$ is defined by either $\aut(\P^2)$ or $\aut_\circ(\PP)$, then $\cS=(r_c)_{c\in\cI_{\P^2}}$ and $\cS=(r_c)_{c\in\cI_{\PP}}$, \resp.
\end{lemma}

\begin{proof}
It follows from \citep[Example~7.1.1]{har} that the biregular automorphisms of projective space are linear.
\end{proof}

\begin{example}
\label{exm:B1}
Let
$f\c\P^2\to\P^3, x \mapsto (x_0^2 + x_1^2 + x_2^2: x_0\,x_1: x_0\,x_2: x_1\,x_2)$
be the parametrization of a \df{Roman surface}~$\img f$ and suppose that $g=f$.
We will see in \EXM{B1b} that $\cS:=(r_c)_{c\in\cI_{\P^2}}$ contains the compatible reparametrizations $\cR(f,g)$.
Since $f$ is base point free, we find that
\[\cJ=\set{c\in\cI_{\P^2}}{ M_{g\circ r_c}\cdot\ker M_f=\zm }.\]
We choose the monomial basis $(x_0^2, x_0x_1, x_0x_2, x_1^2, x_1x_2, x_2^2)$ so that
\begin{gather*}
M_f=
\left(
\begin{smallmatrix}
1 & 0 & 0 & 1 & 0 & 1 \\
0 & 1 & 0 & 0 & 0 & 0 \\
0 & 0 & 1 & 0 & 0 & 0 \\
0 & 0 & 0 & 0 & 1 & 0 \\
\end{smallmatrix}
\right)
,\quad
(\ker f)^\top=
\left(
\begin{smallmatrix}
1 & 0 & 0 & 0 & 0 & -1\\
0 & 0 & 0 & 1 & 0 & -1
\end{smallmatrix}
\right)
\quad\text{and}\quad
M_{g\circ r_c}=
\\
\left(
\begin{smallmatrix}
c_0^2 + c_3^2 + c_6^2 & 2 c_0 c_1 + 2 c_3 c_4 + 2 c_6 c_7 & 2 c_0 c_2 + 2 c_3 c_5 + 2 c_6 c_8 & c_1^2 + c_4^2 + c_7^2 & 2 c_1 c_2 + 2 c_4 c_5 + 2 c_7 c_8 & c_2^2 + c_5^2 + c_8^2\\
c_0 c_3 & c_1 c_3 + c_0 c_4 & c_2 c_3 + c_0 c_5 & c_1 c_4 & c_2 c_4 + c_1 c_5 & c_2 c_5\\
c_0 c_6 & c_1 c_6 + c_0 c_7 & c_2 c_6 + c_0 c_8 & c_1 c_7 & c_2 c_7 + c_1 c_8 & c_2 c_8\\
c_3 c_6 & c_4 c_6 + c_3 c_7 & c_5 c_6 + c_3 c_8 & c_4 c_7 & c_5 c_7 + c_4 c_8 & c_5 c_8
\end{smallmatrix}
\right).
\end{gather*}

We find that $|\cJ|=24$ and 8 elements of $\cJ$ are listed below:
{\footnotesize\begin{gather*}
\begin{array}{ccccccccc}
c_0 & c_1   & c_2  & c_3 & c_4  & c_5  & c_6 & c_7 & c_8 \\\hline
  0 &  \pm1 &    0 &   1 &    0 &    0 &  0  &   0 &  1  \\
  0 &  \pm1 &    0 &  -1 &    0 &    0 &  0  &   0 &  1  \\
  0 &  \pm1 &    0 &   0 &    0 &    1 &  1  &   0 &  0  \\
  0 &  \pm1 &    0 &   0 &    0 &   -1 &  1  &   0 &  0  \\
\end{array}
\end{gather*}}%
We obtain 8 more elements by interchanging columns
$c_0 \leftrightarrow c_1$,
$c_3 \leftrightarrow c_4$ and
$c_6 \leftrightarrow c_7$.
The remaining 8 elements of $\cJ$ are obtained by instead interchanging columns
$c_1 \leftrightarrow c_2$,
$c_4 \leftrightarrow c_5$ and
$c_7 \leftrightarrow c_8$.
If, for example, $c=(0,1,0,1,0,0,0,0,1)$ and $U$ is the matrix \st $U\oplus\im=E_{g\circ r_c} \cdot E_f^{-1}$,
then
\[
U=
\left(
\begin{smallmatrix}
1 & 0 & 0 & 0\\
0 & 1 & 0 & 0\\
0 & 0 & 0 & 1\\
0 & 0 & 1 & 0
\end{smallmatrix}
\right).
\]

Indeed, we verify that $\img\chi_{_U}\circ f\subseteq \img g$, where
\[
\img g=\set{z\in\P^3}{ z_1^2\,z_2^2 + z_1^2\,z_3^2 + z_2^2\,z_3^2 - z_0\,z_1\,z_2\,z_3=0   }.
\]
See \cite{github} for an implementation of this example.
\END
\end{example}

\section{Reduction to five base cases}
\label{sec:cr}

This section will be concluded with the statements of \THM{C} and \THM{B}.
These main results imply a method that reduces the computation of
a super set of compatible reparametrizations to five base cases.
We shall delay the proof of the theorems until \SEC{proofs}.

A \df{condition} is a map $\bc\c\cM\to\{0,1\}$, where $0$ and $1$ are identified
with False and True, \resp.
Let $\cM_\bc:=\set{f\in \cM}{\bc(f)=1}$ be the set of rational maps that satisfy
the condition~$\bc$.
A \df{reducer} consists of a condition~$\bc$ and a
map~$\br\c\cM_\bc\to\cM$.
We refer to $\br(f)$ as the \df{reduction} of~$f$.

We call $\bq\c\cM\to\Z^m$ for some $m\in\Z_{>0}$ a \df{projective invariant}
if $\bq(f)$ is a projective invariant of~$\img f\subset\P^{\edim f}$ for all~$f\in\cM$.
A reducer~$\br$ is \df{compatible}
if its condition~$\bc\c\cM\to\{0,1\}\subset\Z$ is a projective invariant
and if
\[
\cR(f,g)\subseteq \cR(\br(f),\br(g))
\]
for all~$f,g\in\cM_\bc$.

\begin{remark}
\label{rmk:reducer}
Recall from \PRP{S}
that for given birational maps~$f,g\in\cM$
we can recover the projective isomorphisms~$\cP(f,g)$
from a super set of the compatible reparametrizations~$\cR(f,g)$.
Instead of finding a super set of~$\cR(f,g)$ it
is sufficient to find a super set of~$\cR(\br(f),\br(g))$,
where $\br$ is a compatible reducer.
If $\bc(f)\neq \bc(g)$, then $\cP(f,g)=\emptyset$
as the condition~$\bc$ of a compatible reducer~$\br$ is a projective invariant.
If $\bc(f)=\bc(g)=0$,
then we cannot further reduce the problem so that we are in a base case.
In this section we will define three reducers
$\br_0$, $\br_1$, $\br_2$ and we show in \THM{C} that these
reducers are compatible.
The component degree of~$\br_1(f)$
is equal to the component degree of~$f$ minus three (see \PRP{cdeg}).
Hence, it is easier to determine the super set of~$\cR(\br_1(f),\br_1(g))$
instead of~$\cR(f,g)$.
In \THM{B} we classify the base cases, namely the
cases where $\bc_0(f)=\bc_1(f)=\bc_2(f)=0$.
If both $f$ and $g$ are in a base case, then
the computation of a super set
of the compatible reparametrizations is better tractible
as $\img f$ and $\img g$ are surfaces that are covered by lines or conics
and thus theoretically well understood.
\EXM{B2} and \ALG{alg} in the next section shows that our theory can be converted into
an algorithm for computing projective isomorphisms between rational surfaces
in~$\P^n$ for all~$n>1$.
\END
\end{remark}

\begin{proposition}
If $\br\c\cM_\bc\to\cM$ and $\bs\c\cM_\bd\to\cM$ are compatible reducers,
then the reducer $\bs\circ\br$
with condition $\bc(f)\cdot\bd(\br(f))$
for all $f\in\cM$ is compatible as well.
\end{proposition}

\begin{proof}
Straightforward consequence of the definitions.
\end{proof}

Please recall \DEF{nota}.
We define $\bp\c\cM\to\Z^3$
as
\[
\bp(f):=(h^0([f]),~[f]^2,~\gcd [f]).
\]
The reducers $\br_0$, $\br_1$ and $\br_2$
with conditions $\bc_0$, $\bc_1$ and $\bc_2$, \resp,
are defined as follows where $f\in\cM$:
\begin{Mlist}
\item
$\br_0(f):=\Psi_{[f]}$ and $\bc_0(f)$ is defined as
\[
\dim f<h^0([f])-1.
\]

\item
$\br_1(f):=\Psi_c$ with $c:=[f]+\kf$ and $\bc_1(f)$
is defined as
\[
h^0([f]+\kf)>1
\quad\wedge\quad
\neg
\bigl(
[f]^2>\mc{c}^2=c\cdot\mc{c}=0
\bigr).
\]

\item
$\br_2(f):=\Psi_b$ with
$b:=\frac{1}{\gcd[f]}[f]$
and $\bc_2(f)$ is defined as
\[
\gcd[f]>1.
\]
\end{Mlist}

\begin{example}
\label{exm:cdeg2}
Let us consider $f\c\P^2\to\P^3$ in \EXM{B1}
which is defined as
\[
x\mapsto (x_0^2 + x_1^2 + x_2^2: x_0\,x_1: x_0\,x_2: x_1\,x_2).
\]
We have $[f]=2\,\e_0$ and $\bc_0(f)=1$.
The reduction $\br_0(f)\c\P^2\to\P^5$
is up to a choice of basis defined as
\[
x\mapsto (x_0^2 : x_1^2 : x_2^2: x_0\,x_1: x_0\,x_2: x_1\,x_2).
\]
Since $[\br_0(f)]=2\,\e_0$ and $\kappa_{\br_0(f)}=-3\,\e_0$
it follows that $h^0(\br_0(f)+\kappa_{\br_0(f)})=0$ and thus $\bc_1(\br_0(f))=0$.
We verify that $\bc_2(\br_0(f))=1$ as $\gcd [\br_0(f)]=2$
and thus
the reduction $(\br_2\circ\br_0)(f)\c\P^2\to\P^2$
is up to a choice of basis defined as $x\mapsto (x_0:x_1:x_2)$.
\END
\end{example}

\begin{proposition}
\label{prp:cdeg}
$\cdeg \br_i(f)<\cdeg f$
for all $f\in\cM_{\bc_i}$ and $i\in\{1,2\}$.
\end{proposition}

\begin{proof}
Straightforward consequence of the definitions.
\end{proof}

\begin{theorem}
\label{thm:C}
The function $\bp$ is a projective invariant
and the reducers $\br_0$, $\br_1$ and $\br_2$ are compatible.
\end{theorem}

\begin{definition}
We say that $f\in\cM$ is characterized by a \df{base case}
if
\[
\bc_0(f)=\bc_1(f)=\bc_2(f)=0
\]
and either one of the following five cases holds:
\begin{Mlist}
\item {\bf B1.} $h^0([f])=3$, $[f]^2=1$ and $\img f\cong\P^2$.
\item {\bf B2.} $h^0([f])=4$, $[f]^2=2$ and $\img f$ is a quadric surface.
\item {\bf B3.} $h^0([f])=[f]^2+1$, $1\leq [f]^2\leq 8$ and $\img \Psi_{\alpha [f]}$ is a del Pezzo surface, where $\alpha:=\max(4-[f]^2,1)$. Moreover, the map $\Psi_{\alpha [f]}$ is birational and if $[f]^2\geq 3$, then $\img \Psi_{\alpha [f]}$ is covered by conics.
\item {\bf B4.} $h^0(2\,[f]+\kf)\geq 2$, $\mc{2\,[f]+\kf}^2=0$ and $\img f$  is a surface covered by lines.
\item {\bf B5.} $h^0([f]+\kf)\geq 2$, $\mc{[f]+\kf}^2=0$ and $\img f$ is a surface covered by conics or lines.
\END
\end{Mlist}
\end{definition}

\begin{example}
\label{exm:B4}
Suppose that $f\in\cM$ is a rational map
\st $\dom f=\P^2$ and $[f]=8\,\e_0-5\,\e_1-3\,\e_2-3\,\e_3$.
We set $u:=[f]+\kappa_f=5\,\e_0-4\,\e_1-2\,\e_2-2\,\e_3$.
Since~$u$ is negative against
the classes of lines defined by $\e_0-\e_1-\e_2$ and $\e_0-\e_1-\e_4$
it follows that $u\neq\mc{u}$.
We subtract these classes from $u$
and find that $[\br_1(f)]=\mc{u}=3\,\e_0-2\e_1-\e_2-\e_3$ (see \EXM{cdeg3}).
We notice that $\bc_0(\br_1(f))=\bc_1(\br_1(f))=\bc_2(\br_1(f))=0$,
since $[\br_1(f)]+\k_{\br_1(f)}=-\e_1$ and $\gcd [\br_1(f)]=1$.
The class $v:=2\,[\br_1(f)]+\k_{\br_1(f)}=3\,\e_0-3\,\e_1-\e_2-\e_3$
is negative against the classes
$\e_0-\e_1-\e_2$ and $\e_0-\e_1-\e_4$.
Therefore, $\mc{v}=\e_0-\e_1$ and
thus $\br_1(f)$ is characterized by base case~B4.
\END
\end{example}

Recall \RMK{reducer} for the purpose of the following theorem.

\begin{theorem}
\label{thm:B}
Suppose that $f\in \cM$ \st $\bc_0(f)=\bc_1(f)=\bc_2(f)=0$.
\begin{Mclaims}
\item[\bf a)]
If $[f]^2>0$, then $f$ is characterized by base case B1, B2, B3, B4 or B5.

\item[\bf b)]
If $[f]^2=0$, then
there exists no $g\in\cM_{\bc_1}$ \st
$[g]^2>0$ and
$f\in\{\br_1(g),~ (\br_2\circ\br_1)(g) \}$.

\end{Mclaims}
\end{theorem}

\begin{example}
If $f,g\in\cM$ \st
$[f]=3\,\e_0-\e_1-\cdots-\e_9$
and
$[g]=6\,\e_0-2\,\e_1-\cdots-2\,\e_9$,
then $h^0([f])=2$, $[f]^2=0$ and $f=\br_1(g)$.
However, $[g]^2=0$ and thus this is not a counter example to \THM{B}b.
\END
\end{example}

\section{Reparametrizations for B1 and B2}
\label{sec:B}

In this section we characterize a super set for~$\cR(f,g)$
in case $f,g\in\cM$ are characterized by base case B1 or B2.
The main results of this article are translated into \ALG{alg}.

\begin{proposition}[B1]
\label{prp:B1}
If $f,g\in\cM$ are both characterized by B1, then
\[
\cR(f,g)
~\subseteq~
\cS:=\set{ g^{-1}\circ r_c \circ f }{ c\in \cI_{\P^2} }.
\]
\end{proposition}

\begin{proof}
By \THM{B} we have
$\img f\cong\P^2\cong \img g$ and $\aut(\P^2)$
is characterized in \LEM{aut}.
\end{proof}

\begin{example}[B1]
\label{exm:B1b}
Let $f\c \P^2\to \P^3$ be defined as in \EXM{cdeg2}
and recall that~$(\br_2\circ\br_0)(f)\c\P^2\to\P^2$
is a projective isomorphism.
Thus if $f=g$,
then it follows from \PRP{B1} that
$(r_c)_{c\in\cI_{\P^2}}$
is a super set of $\cR((\br_2\circ\br_0)(f),(\br_2\circ\br_0)(g))$
and thus also a super set of $\cR(f,g)$.
\END
\end{example}

\begin{proposition}[B2]
\label{prp:B2}
Suppose that $f,g\in\cM$ are characterized by base case~B2.
We consider the following \df{set of classes of lines} that are contained in $\img f$:
\[
\bF(f):=\set{c\in N(\bmd f)}{c^2=0,~ [f]\cdot c=1,~ c=\mc{c}}.
\]
\begin{Mclaims}
\item[\bf a)]
If $|\bF(f)|\geq 2$, then $\bF(f)=\{a,b\}$, $\bF(g)=\{u,v\}$ and
\begin{multline*}
\cR(f,g)
~~\subseteq~~
\cS:=
\set{ (\Psi_u\times\Psi_v)^{-1} \circ r_c\circ (\Psi_a\times\Psi_b) }{c\in \cI_{\PP}}
\\
\cup~~
\set{ (\Psi_u\times\Psi_v)^{-1} \circ r_c\circ (\Psi_b\times\Psi_a) }{c\in \cI_{\PP}}.
\end{multline*}

\item[\bf b)]
If $|\bF(f)|\leq 1$, then
$\img f$ and $\img g$ are quadric cones
and
\[
\cR(f,g)
~~\subseteq~~
\cS:=\set{g^{-1}\circ t^{-1}\circ r_c\circ s\circ f}{c\in \cI}
\]
where
$\cI:=\set{c\in\F^{16}}{r_c\in \aut(\P^3) \text{ and } r_c(Q)=Q}$
\st
$Q\subset\P^3$ is some fixed quadric cone in diagonal form
and
$s,t\in \aut(\P^3)$ are projective isomorphisms
\st $s(\img f)=t(\img g)=Q$.
\end{Mclaims}
\end{proposition}

\begin{proof}
a)
We observe that $\bF(f)=\{a,b\}$
where $a$ and $b$ are classes of intersecting lines on $\img f$ so that $h^0(a)=h^0(b)=2$ and $a\cdot b=1$.
Since $a^2=0$, we obtain after resolving its base points a morphism $\xi_a\c\bmd f\to\P^1$
and thus the composition $\Psi_a\circ f^{-1}\c\img f\to\P^1$ is a morphism as well.
Moreover, the following maps are birational morphisms
\begin{gather*}
\mu:=(\Psi_a\times\Psi_b)\circ f^{-1}\c\img f\to \PP,
\\
\nu:=(\Psi_u\times\Psi_v)\circ g^{-1}\c\img g\to \PP.
\end{gather*}
Indeed, $a\cdot b=1$ and
thus the preimages $\mu^{-1}(\{t\}\times\P^1)$
and $\mu^{-1}(\P^1\times \{t'\})$
are lines that intersect at the unique point in $\img f$ whose image is $(t,t')\in\PP$
for all $t,t'\in\P^1$.
Now let the line $L\subset \img f$ be defined by $f\big(\Psi_a^{-1}(t)\big)$
for some $t\in\P^1$.
For all $p\in\cP(f,g)$ there exists $t'\in\P^1$ and $u\in\bF(g)$ \st
$p(L)\subset\img g$ is defined by $g\big(\Psi_u^{-1}(t')\big)$.
It follows that for all $p\in\cP(f,g)$ there exists a biregular map
$r\c\P^1\to\P^1$ \st
$r\circ\Psi_a\circ f^{-1}=\Psi_u\circ g^{-1}\circ p$.
Therefore, for all $p\in\cP(f,g)$
there exists $r\in\aut(\PP)$
\st $r\circ\mu=\nu\circ p$.
By interchanging $a$ and $b$
we may assume \Wlog that $r\in\aut_\circ(\PP)$ so that the factors of $\PP$ are not flipped by $r$.
The proof is now concluded with \LEM{aut}.

For assertion b) we consider the symmetric matrices defined by the quadratic forms
associated to $\img f,\img g\subset\P^3$.
We use matrix diagonalization on these matrices to compute $Q\subset\P^3$ and $s,t\in \aut(\P^3)$.
\end{proof}

The following example for case B2 explains how to compute $\cS$ in case the
quadrics are doubly ruled.

\begin{example}[B2]
\label{exm:B2}
Suppose that we are given the following birational maps:
\begin{gather*}
f\c\P^2\dto X\subset\P^3,\quad (x_0:x_1:x_2)\mapsto
(x_0^6x_1^2: x_0x_1^5x_2^2: x_1^3x_2^5: x_0^5x_1x_2^2 + 2x_0^5x_2^3),
\\
g\c\PP\dto Y\subset\P^3,\quad (y_0:y_1;y_2:y_3)\mapsto
(y_0^3y_1^2y_2^5:
y_0^3y_1^2y_2^5 + y_1^5y_2^3y_3^2:
\\
y_0^2y_1^3y_3^5:
y_0^4y_1y_2^3y_3^2 + y_0^5y_2^2y_3^3 + y_0^2y_1^3y_3^5).
\end{gather*}
Our goal is to compute the projective isomorphisms $\cP(f,g)$.
We use \ALG{get} to compute
the base points of the linear series associated to~$f$ and $g$.
We find that $f$ has simple base points at
$p_1:=(0:0:1)$,
$p_2:=(0:1:0)$ and
$p_3:=(1:0:0)$
with multiplicities $3$, $2$ and $2$, \resp.
The infinitely near relations between the
remaining 10 base points $p_4,\ldots,p_{13}$ of $f$ are as follows:
\begin{gather*}
p_7\rightsquigarrow p_6 \rightsquigarrow p_4 \rightsquigarrow p_1,
\quad
p_9\rightsquigarrow p_8 \rightsquigarrow p_5 \rightsquigarrow p_2,
\\
p_{10}\rightsquigarrow p_{11} \rightsquigarrow p_2,
\quad
p_{13}\rightsquigarrow p_{12} \rightsquigarrow p_3.
\end{gather*}
The simple base points of $g$ are
$q_1:=(0:1;0:1)$, $q_2:=(0:1;1:0)$, $q_3:=(1:0;0:1)$, $q_4:=(1:0;1:0)$
and have each multiplicity $2$.
The infinitely near relations between the
remaining 8 base points $q_5,\ldots,q_{12}$ of $g$ are as follows:
\begin{gather*}
q_6\rightsquigarrow q_5\rightsquigarrow q_1,\quad
q_8\rightsquigarrow q_7\rightsquigarrow q_2,\quad
q_{10}\rightsquigarrow q_9\rightsquigarrow q_3,\quad
q_{12}\rightsquigarrow q_{11}\rightsquigarrow q_4.
\end{gather*}
The multiplicities of the base points are encoded by the
classes of~$f$ and $g$:
\begin{gather*}
[f]=8\,\e_0-3\,\e_1-3\,\e_2-2\,\e_3-2\,\e_4-2\,\e_5-\e_6-\ldots-\e_{13}
\quad\text{and}
\\
[g]=5\,\l_0+5\,\l_1-2\,\p_1-2\,\p_2-2\,\p_3-2\,\p_4-\p_5-\ldots-\p_{12}.
\end{gather*}
We observe that $[f]^2=[g]^2=\deg X=\deg Y=26$.
We apply \ALG{set}
and find that $h^0([f])=h^0([g])=16$
and thus $(\bc_0(f),\bc_0(g))=(1,1)$ since $\dim f= \dim g=3<16-1$.

We set $(\hf,\hg):=(\br_0(f),\br_0(g))$
so that $\bp(\hf)=\bp(\hg)=(16,26,1)$.
Notice that $[\hf]=[f]$ and $[\hg]=[g]$
as a direct consequence of the definitions.
We find that $(\bc_1(\hf),\bc_1(\hg))=(1,1)$ and $h^0(\hf+\khf)=h^0(\hg+\khg)=12$.

We set $(\hf,\hg):=(\br_1(\hf),\br_1(\hg))$ so that
$\bp(\hf)=\bp(\hg)=(12,14,1)$ and
\begin{gather*}
[\hf]=5\,\e_0-2\,\e_1-2\,\e_2-\e_3-\e_4-\e_5
\quad\text{and}\\
[\hg]=3\,\l_0+3\,\l_1-\p_1-\p_2-\p_3-\p_4.
\end{gather*}
We remark that
$[f]+\kf=\mc{[f]+\kf}$ and $[g]+\kg=\mc{[g]+\kg}$, unlike as we have seen in \EXM{B4}.
We find that $(\bc_1(\hf),\bc_1(\hg))=(1,1)$ and $h^0(\hf+\khf)=h^0(\hg+\khg)=4$.

We set $(\hf,\hg):=(\br_1(\hf),\br_1(\hg))$ so that
$\bp(\hf)=\bp(\hg)=(4,2,1)$ and
\begin{gather*}
[\hf]=2\,\e_0-\e_1-\e_2
\quad\text{and}\quad
[\hg]=\l_0+\l_1.
\end{gather*}
We verify that $h^0(\hf+\khf)=h^0(\hg+\khg)=0$
so that $(\bc_1(f),\bc_1(g))=(0,0)$.
It follows from \THM{B} that $f$ and $g$
are characterized by base case~B2.

We find that $\bF(\hf)=\{a,b\}$ and $\bF(\hg)=\{u,v\}$ as defined at \PRP{B2},
where $a:=\e_0-\e_1$, $b:=\e_0-\e_2$, $u:=\l_0$ and $v:=\l_1$
so that
\begin{gather*}
\Psi_{a}\times\Psi_{b}\c\P^2\dto\PP,\quad x\mapsto (x_0:x_1;x_0:x_2),
\\
\Psi_{b}\times\Psi_{a}\c\P^2\dto\PP,\quad x\mapsto (x_0:x_2;x_0:x_1),\quad \text{and}
\\
\Psi_{u}\times\Psi_{v}\c \PP\to\PP,\quad y\mapsto y.
\end{gather*}
We consider the following reparametrizations $\P^2\dto\PP$:
\begin{gather*}
s_c:x\mapsto
(
c_0\,x_0 + c_1\,x_1:
c_2\,x_0 + c_3\,x_1;
c_4\,x_0 + c_5\,x_2:
c_6\,x_0 + c_7\,x_2
),
\\
t_c:x\mapsto
(
c_0\,x_0 + c_1\,x_2:
c_2\,x_0 + c_3\,x_2;
c_4\,x_0 + c_5\,x_1:
c_6\,x_0 + c_7\,x_1
),
\end{gather*}
and we set
$\cS:=\{s_c\}_{c\in \cI_{\PP}}\cup\{t_c\}_{c\in \cI_{\PP}}$.
It follows from \PRP{B2} that $\cS\supseteq\cR(f,g)$.

Let us first consider the reparametrizations $\{s_c\}_{c\in \cI_{\PP}}$.
For general $c\in\cI_{\PP}$ we observe that $\cdeg(g\circ s_c)=10$, although $\cdeg( f)=8$.
We use \citep[Algorithm~2]{n-bp} to compute
\[
\cJ':=\set{c\in\cI}{ g\circ s_c \text{ has the same base points as } f },
\]
and find that
$\cJ'=\cJ_{0}\cup\cJ_{1}$, where
\begin{gather*}
\cJ_{0}=\set{c\in \cI_{\PP}}{ c_0=c_4=1,~ c_1=c_2=c_5=c_6=0 },
\\
\cJ_{1}=\set{c\in \cI_{\PP}}{ c_1=c_5=1,~ c_0=c_3=c_4=c_7=0 }.
\end{gather*}
If we substitute $s_c$ into $g$ for $c\in\cJ'$, then
the greatest common divisor of the components is $x_0^2$.
Therefore $\cdeg(g\circ s_c)=8=\cdeg(f)$ as required.
Next we enforce that the $4\times 45$ coefficient matrix $M_{g\circ s_c}$
has the same kernel as the coefficient matrix $M_f$ and compute the corresponding index sets:
\begin{gather*}
\cJ_0'':=\set{ c\in \cJ_0}{  M_{g\circ r_c}\cdot \ker M_f=\zm }=\set{c\in \cJ'}{ c_7=2\,c_3 }\quad\text{and}
\\
\cJ_1'':=\set{ c\in \cJ_1}{  M_{g\circ r_c}\cdot \ker M_f=\zm }=\emptyset.
\end{gather*}
Next, we perform the same procedure for $\{t_c\}_{c\in \cI_{\PP}}$, but
in this case we arrive at only empty-sets.
Therefore, it follows that $\cJ$ as defined at~\EQN{J} is equal to~$\cJ_0''$.
We apply \PRP{S} and recover $\cP(f,g)$ in terms of a matrix parametrized in terms of $c_3\neq 0$:
{\footnotesize%
\[
U:=
\begin{bmatrix}
1 & 0        & 0         &      0\\
1 & 4\,c_3^5 & 0         &      0\\
0 & 0        & 32\,c_3^6 &      0\\
0 & 0        & 32\,c_3^6 & 4\,c_3
\end{bmatrix}.
\]
}%
Indeed, if we substitute $\chi_{_U}\circ f$ with indeterminate $c_3$ into the equation of $Y\subset \P^3$,
then we obtain 0.
See \cite{github} for an implementation of this example.
\END
\end{example}

\begin{remark}
The current state is summarized in \ALG{alg}
whose correctness follows from \THM{C} and \THM{B}.
Recall that in case B3, B4 and B5, the surface~$\img \hf$
is covered by lines or conics.
Such surfaces are theoretically well understood,
but a complete algorithmic description is left as future work.
\END
\end{remark}

\begin{algorithm}
\caption{}
\label{alg:alg}
\begin{itemize}[itemsep=0pt,topsep=5pt,leftmargin=5mm]

\item {\bf Input.}
Birational maps $f,g\in\cM$.

\item {\bf Output.} The set of projective isomorphisms $\cP(f,g)$.

\item {\bf Method.} {\it We use the \# symbol for comments.}
\item[] Compute the base points of $f$ and $g$ with \ALG{get}.
\item[] {\bf if\quad}$\bp(f)\neq\bp(g) ${\bf\quad then\quad}{\bf return\quad}$\emptyset$;
\item[] $(\hf,\hg):=(f,g)$;
\item[] {\bf if\quad}$(\bc_0(\hf),\bc_0(\hg))=(1,1)${\bf\quad then\quad} $(\hf,\hg) :=(\br_0(\hf),\br_0(\hg))$;

\item[]
{\bf while \quad}
$(\bc_1(\hf),\bc_1(\hg))=(1,1)$
{\bf\quad do:\quad}
\begin{itemize}[itemsep=0pt,topsep=0pt,leftmargin=10mm]
\item[] $(\hf,\hg) :=(\br_1(\hf),\br_1(\hg))$;
\end{itemize}
\item[] {\bf if\quad}$(\bc_2(\hf),\bc_2(\hg))=(1,1)${\bf\quad then\quad} $(\hf,\hg) :=(\br_2(\hf),\br_2(\hg))$;
\item[] {\bf if\quad}$\bp(\hf)\neq\bp(\hg) ${\bf\quad then\quad}{\bf return\quad}$\emptyset$;
\item[] {\bf if\quad}$h^0([\hf])=3${\bf\quad and\quad}$[\hf]^2=1${\bf\quad then\quad}{\it\# case B1}
\begin{itemize}[itemsep=0pt,topsep=0pt,leftmargin=10mm]
\item[] Set $\cS$ as defined in \PRP{B1}.
\end{itemize}

\item[] {\bf else if\quad}$h^0([\hf])=4${\bf\quad and\quad}$[\hf]^2=2${\bf\quad then\quad}{\it\# case B2}
\begin{itemize}[itemsep=0pt,topsep=0pt,leftmargin=10mm]
\item[] Set $\cS$ as defined in \PRP{B2}.
\end{itemize}

\item[] {\bf else if\quad}$h^0([\hf])=[\hf]^2+1${\bf\quad and\quad}$1\leq [\hf]^2\leq 8${\bf\quad then\quad}{\it\# case B3}
\begin{itemize}[itemsep=0pt,topsep=0pt,leftmargin=10mm]
\item[] {\it \# Not considered in this article.}
\end{itemize}

\item[] {\bf else if\quad}$h^0(2\,[\hf]+\khf)\geq 2${\bf\quad and\quad}$\mc{2\,[\hf]+\khf}^2=0${\bf\quad then\quad}{\it\# case B4}
\begin{itemize}[itemsep=0pt,topsep=0pt,leftmargin=10mm]
\item[] {\it \# Not considered in this article.}
\end{itemize}

\item[] {\bf else if\quad}$h^0([\hf]+\khf)\geq 2${\bf\quad and\quad}$[\mc{[\hf]+\khf}^2=0${\bf\quad then\quad}{\it\# case B5}
\begin{itemize}[itemsep=0pt,topsep=0pt,leftmargin=10mm]
\item[] {\it \# Not considered in this article.}
\end{itemize}

\item[] {\bf else:\quad}{\bf return\quad}$\emptyset$;

\item[] Compute $\cP(f,g)$ from $\cS\supseteq\cR(f,g)$ using \PRP{S}.

\item[] {\bf return\quad}$\cP(f,g)$;
\end{itemize}
\end{algorithm}

\section{Applications of the algorithm}
\label{sec:A}

In this section we outline how
to find the projective isomorphism
that correspond to
(non-) Euclidean isomorphisms between rational surfaces.

The \df{M\"obius quadric} is defined as
\[
\S^n:=\set{ x\in\P^{n+1} }{ -x_0^2+x_1^2+\ldots+x_{n+1}^2=0 }.
\]
Let $\pi\c \S^n\dto \P^n$
be the \df{stereographic projection} defined as
\[
\pi\c\S^n\dto\P^n,\qquad (x_0:\ldots:x_{n+1}) \mapsto (x_0-x_{n+1}:x_1:\ldots:x_n),
\]
with inverse $\pi^{-1}\c \P^n\dto \S^n$,
\[
z
\mapsto
(z_0^2+z_1^2+\ldots+z_n^2:2z_0z_1:\ldots:2z_0z_n: -z_0^2+z_1^2+\ldots+z_n^2).
\]

Suppose that $f,g\in\cM$ are birational.

The \df{affine isomorphisms}
between $\img f$ and $\img g$
are defined as
\[
\set{\rho\in\cP(f,g)}{\rho(\H_n)=\H_n},
\]
where $\H_n:=\set{z\in\P^n}{z_0=0}$.

The \df{Euclidean isomorphisms}
between $\img f$ and $\img g$
are defined as
\[
\set{\rho\in\cP(f,g)}{\rho(\E_n)=\E_n},
\]
where $\E_n:=\set{z\in\H_n}{z_0^2+\ldots+z_n^2=0}$.
We remark that Euclidean isomorphisms are perhaps better known as ``Euclidean similarities''.

The \df{M\"obius isomorphisms} between $\img f$ and $\img g$ are defined as
\[
\set{\pi \circ \rho \circ \pi^{-1}}{ \rho\in \cP(\pi^{-1}\circ f,\pi^{-1}\circ g) \text{ and } \rho(\S^n)=\S^n}.
\]
Notice that if $\alpha\c \P^n\dto \P^n$ is a M\"obius isomorphism \st $\alpha(\img f)=\img g$,
then $\alpha$ is a birational quadratic map \st $\alpha(\E_n)=\E_n$.

It is straightforward to recover the affine, Euclidean, or M\"obius isomorphisms from the output
$\cP(f,g)$ or $\cP(\pi^{-1}\circ f,\pi^{-1}\circ g)$ of \ALG{alg}.

\begin{example}
We continue with \EXM{B1}, where $f$ parametrizes a Roman surface
and where the projective isomorphism $\chi_{_U}\in \cP(f,f)$
corresponding to $c=(0,1,0,1,0,0,0,0,1)$ is
defined as $\chi_{_U}(x)=(x_0:x_1:x_3:x_2)$.
We check that $\chi_{_U}(\E_3)=\E_3$ and thus $\chi_{_U}$ is an Euclidean isomorphism.
In fact, we verify that $|\cP(f,f)|=|\set{\rho\in\cP(f,f)}{\rho(\E_3)=\E_3}|=24$
and thus the Roman surface admits 24 Euclidean symmetries, namely the symmetries
of a tetrahedron.
\END
\end{example}

\section{The proofs of the theorems}
\label{sec:proofs}

In this section we prove \THM{C} and \THM{B}.
We assume that the reader
is familiar with the material of \citep[Sections~II.7 and V.3]{har}
and \citep[Chapter~1]{mat}.

\begin{definition}
\label{def:map2}
Suppose that $Z$ is a rational surface
and recall that on rational surfaces the numerical- and rational- equivalence
relations for divisor classes are the same.
Let $c\in N(Z)$ be a class \st $h^0(c)>0$ and
let $V:=H^0(Z,c)$ denote the vector space of global sections over the ground field~$\F$.
We define $\varphi_c\c Z\to \P^{h^0(c)-1}$ as $\varphi_V$ as defined in \DEF{map}.
\END
\end{definition}

\begin{proof}[Proof of \THM{C}.]
We consider the birational morphism
$\pi\c\bmd f\to \dom f$ that resolves the base locus of~$f$
so that the composition $\Psi_{[f]}\circ\pi\c\bmd f\to \img f$ is a morphism.
Notice that $[f]\in N(\bmd f)$
is the divisor class of the pullback of a hyperplane section along this morphism
and that $\kf$ is the canonical class of~$\bmd f$.
We recall \DEF{map2} and notice that for all $c\in N(\bmd f)$
the map~$\varphi_c$ is up to a choice of basis equivalent to the morphism $\Psi_c\circ\pi$.
In particular, we will notice that the diagrams in this proof
remain commutative when an arrow for $\pi\c\bmd f\to\dom f$ is included.

The case $\cR(f,g)=\emptyset$ is trivial
and therefore we will assume that $\cR(f,g)\neq \emptyset$.
Suppose that $\gamma\in\cR(f,g)$ is an arbitrary but fixed
compatible reparametrization and
let $\beta\in\cP(f,g)$ be a projective isomorphism \st $\beta\circ f=g\circ\gamma$.

First we observe that $\cR(\br_i(f),\br_i(g))$ does not depend on the choice
of basis of the associated maps $\br_i(f)$ and $\br_i(g)$
for all $i\in\{0,1,2\}$ (see \DEF{map}).

In order to show that $\br_0$ is compatible we need to show that
the condition~$\bc_0$ is a projective invariant and that
$\gamma\in\cR(\br_0(f),\br_0(g))$ if $(\bc_0(f),\bc_0(g))=(1,1)$.

Let $\hX:=\img\br_0(f)$, $\hY:=\img\br_0(g)$ and recall that $\br_0(f)=\Psi_{[f]}$.
As a consequence of the definitions,
there exists a birational and degree preserving linear projection~$\rho_f\c \hX\to \img f$
so that $\rho_f\circ\Psi_{[f]}=f$.
Let $\alpha\c \bmd f\dto \bmd g$ be the birational map that makes the
diagram of \FIG{r0a} commutative.

\begin{figure}[!ht]
\centering
\begin{tikzpicture}[node distance=13mm, auto]
\usetikzlibrary{arrows.meta}
\node (C1) {};
\node (A) [right of=C1] {$\bmd f$};
\node (C3) [right of=A] {};
\node (S) [right of=C3]  {};
\node (C5) [right of=S]  {};
\node (C6) [right of=C5] {};
\node (C7) [right of=C6] {};
\node (B) [right of=C7] {$\bmd g$};

\node (X)  [below of=A] {$\hX$};
\node (E2) [right of=X]  {};
\node (hX) [right of=E2] {$\img f$};
\node (E4) [right of=hX]  {};
\node (hY) [right of=E4] {$\img g$};
\node (E6) [right of=hY]  {};
\node (Y)  [right of=E6] {$\hY$};
\node (Df) [below of=E2] {$\dom f$};
\node (Dg) [below of=E6] {$\dom g$};

\draw[->] (A) to node [swap] {$\varphi_{[f]}$} (X);
\draw[->] (B) to node {$\varphi_{[g]}$} (Y);
\draw[->, dashed] (Df) to node {$\br_0(f)$} (X);
\draw[->, dashed] (Dg) to node [swap] {$\br_0(g)$} (Y);


\draw[->, dashed] (Df) to node {$\gamma$} (Dg);

\draw[->, dashed] (Df) to node [swap] {$f$} (hX);
\draw[->, dashed] (Dg) to node {$g$} (hY);

\draw[->] (X) to node {$\rho_f$} (hX);
\draw[->] (Y) to node [swap] {$\rho_g$} (hY);
\draw[->, dashed] (A) to node {$\alpha$} (B);

\draw[->] (hX) to node {$\beta$} (hY);
\end{tikzpicture}
\caption{See proof of \THM{C}.}
\label{fig:r0a}
\end{figure}

It follows from the factorization theorem for birational maps \citep[Corollary~1-8-4]{mat}
that there exists a smooth surface $S$ and birational morphisms
$s\c S\to \bmd f$ and $t\c S\to \bmd g$
so that the diagram in \FIG{r0b} commutes.

\begin{figure}[!ht]
\centering
\begin{tikzpicture}[node distance=13mm, auto]
\usetikzlibrary{arrows.meta}
\node (C1) {};
\node (C2) [right of=C1] {};
\node (C3) [right of=C2] {};
\node (S) [right of=C3]  {$S$};
\node (C5) [right of=S]  {};
\node (C6) [right of=C5] {};
\node (C7) [right of=C6] {};

\node (A) [below of=C1] {$\bmd f$};;
\node (B) [below of=C7] {$\bmd g$};

\node (X)  [below of=A] {$\hX$};
\node (E2) [right of=X]  {};
\node (hX) [right of=E2] {$\img f$};
\node (E4) [right of=hX]  {};
\node (hY) [right of=E4] {$\img g$};
\node (E6) [right of=hY]  {};
\node (Y)  [right of=E6] {$\hY$};
\node (Df) [below of=E2] {$\dom f$};
\node (Dg) [below of=E6] {$\dom g$};

\draw[->] (A) to node [swap] {$\varphi_{[f]}$} (X);
\draw[->] (B) to node {$\varphi_{[g]}$} (Y);
\draw[->, dashed] (Df) to node {$\br_0(f)$} (X);
\draw[->, dashed] (Dg) to node [swap] {$\br_0(g)$} (Y);

\draw[->, red] (S) to node [swap] {$s$} (A);
\draw[->, red] (S) to node {$t$} (B);

\draw[->, dashed] (Df) to node {$\gamma$} (Dg);

\draw[->, dashed] (Df) to node [swap] {$f$} (hX);
\draw[->, dashed] (Dg) to node {$g$} (hY);
\draw[->] (X) to node {$\rho_f$} (hX);
\draw[->] (Y) to node [swap] {$\rho_g$} (hY);
\draw[->, dashed] (A) to node {$\alpha$} (B);
\draw[->] (hX) to node {$\beta$} (hY);
\end{tikzpicture}
\caption{See proof of \THM{C}.}
\label{fig:r0b}
\end{figure}

Since $s^*[f]=t^*[g]$ we find that $\hX=\img\varphi_{s^*[f]}$ is
projectively isomorphic to $\hY=\img\varphi_{t^*[g]}$.
Hence,
$\bc_0$ is a projective invariant and
there exists a projective isomorphism
$\hat{\beta}\c \hX\to \hY$
that makes the diagram in \FIG{r0c} commutative.
It follows that
$\gamma\in\cR(\br_0(f),\br_0(g))$
and thus $\br_0$ is a compatible reducer as asserted.

\begin{figure}[!ht]
\centering
\begin{tikzpicture}[node distance=13mm, auto]
\usetikzlibrary{arrows.meta}
\node (C1) {};
\node (C2) [right of=C1] {};
\node (C3) [right of=C2] {};
\node (S) [right of=C3]  {$S$};
\node (C5) [right of=S]  {};
\node (C6) [right of=C5] {};
\node (C7) [right of=C6] {};

\node (A) [below of=C1] {$\bmd f$};;
\node (B) [below of=C7] {$\bmd g$};

\node (X)  [below of=A] {$\hX$};
\node (E2) [right of=X]  {};
\node (hX) [right of=E2] {};
\node (E4) [right of=hX]  {};
\node (hY) [right of=E4] {};
\node (E6) [right of=hY]  {};
\node (Y)  [right of=E6] {$\hY$};
\node (Df) [below of=E2] {$\dom f$};
\node (Dg) [below of=E6] {$\dom g$};

\draw[->] (A) to node [swap] {$\varphi_{[f]}$} (X);
\draw[->] (B) to node {$\varphi_{[g]}$} (Y);
\draw[->, dashed] (Df) to node {$\br_0(f)$} (X);
\draw[->, dashed] (Dg) to node [swap] {$\br_0(g)$} (Y);

\draw[->] (S) to node [swap] {$s$} (A);
\draw[->] (S) to node {$t$} (B);

\draw[->, dashed] (Df) to node {$\gamma$} (Dg);

\draw[->, dashed,red] (S) to node {$\varphi_{s^*[f]}$} (X);
\draw[->, dashed,red] (S) to node [swap] {$\varphi_{t^*[g]}$} (Y);

\draw[->,red] (X) to node {$\hat{\beta}$} (Y);

\end{tikzpicture}
\caption{See proof of \THM{C}.}
\label{fig:r0c}
\end{figure}

In the remainder of the proof we assume that $f=\br_0(f)=\Psi_{[f]}$.
Since $\br_0$ is compatible, this assumption is without loss of generality.

In order to show that $\br_2$ is compatible we need to show that
$\bc_2$ is a projective invariant and that
$\gamma\in\cR(\br_2(f),\br_2(g))$ if $(\bc_2(f),\bc_2(g))=(1,1)$.

As before, let $\alpha\c\bmd f\dto\bmd g$ be a birational map
\st
$\beta\circ \varphi_{[f]}=\varphi_{[g]}\circ\alpha$.
It follows from the factorization theorem for birational maps
that there exists a smooth surface $S$ and birational morphisms
$s\c S\to \bmd f$ and $t\c S\to \bmd g$
so that the diagram in \FIG{r2a} commutes.

\begin{figure}[!ht]
\centering
\begin{tikzpicture}[node distance=13mm, auto]
\usetikzlibrary{arrows.meta}
\node (C1) {};
\node (C2) [right of=C1] {};
\node (C3) [right of=C2] {};
\node (S) [right of=C3]  {$S$};
\node (C5) [right of=S]  {};
\node (C6) [right of=C5] {};
\node (C7) [right of=C6] {};

\node (A) [below of=C1] {$\bmd f$};;
\node (B) [below of=C7] {$\bmd g$};

\node (X)  [below of=A] {$\img f$};
\node (E2) [right of=X]  {};
\node (hX) [right of=E2] {};
\node (E4) [right of=hX] {};
\node (hY) [right of=E4] {};
\node (E6) [right of=hY] {};
\node (Y)  [right of=E6] {$\img g$};
\node (Df) [below of=E2] {$\dom f$};
\node (Dg) [below of=E6] {$\dom g$};

\draw[->] (A) to node [swap] {$\varphi_{[f]}$} (X);
\draw[->] (B) to node {$\varphi_{[g]}$} (Y);
\draw[->, dashed] (Df) to node {$f$} (X);
\draw[->, dashed] (Dg) to node [swap] {$g$} (Y);

\draw[->, red] (S) to node [swap] {$s$} (A);
\draw[->, red] (S) to node {$t$} (B);
\draw[->, dashed] (A) to node {$\alpha$} (B);
\draw[->] (X) to node {$\beta$} (Y);
\draw[->, dashed] (Df) to node {$\gamma$} (Dg);
\end{tikzpicture}
\caption{See proof of \THM{C}.}
\label{fig:r2a}
\end{figure}

Let $\hX:=\img\br_2(f)$, $\hY:=\img\br_2(g)$,
$\fa:=\frac{1}{\gcd f}[f]$ and $\fb:=\frac{1}{\gcd g}[g]$.
We consider the diagram of \FIG{r2b} that is commutative as a direct
consequence of the definitions.

\begin{figure}[!ht]
\centering
\begin{tikzpicture}[node distance=13mm, auto]
\usetikzlibrary{arrows.meta}
\node (C1) {};
\node (C2) [right of=C1] {};
\node (C3) [right of=C2] {};
\node (S) [right of=C3]  {$S$};
\node (C5) [right of=S]  {};
\node (C6) [right of=C5] {};
\node (C7) [right of=C6] {};

\node (A) [below of=C1] {$\bmd f$};;
\node (B) [below of=C7] {$\bmd g$};

\node (X)  [below of=A] {$\img f$};
\node (E2) [right of=X]  {};
\node (hX) [right of=E2] {$\hX$};
\node (E4) [right of=hX]  {};
\node (hY) [right of=E4] {};
\node (E6) [right of=hY]  {};
\node (Y)  [right of=E6] {$\img g$};
\node (Df) [below of=E2] {$\dom f$};
\node (Dg) [below of=E6] {$\dom g$};

\draw[->] (A) to node [swap] {$\varphi_{[f]}$} (X);
\draw[->] (B) to node {$\varphi_{[g]}$} (Y);
\draw[->, dashed] (Df) to node {$f$} (X);
\draw[->, dashed] (Dg) to node [swap] {$g$} (Y);

\draw[->] (S) to node [swap] {$s$} (A);
\draw[->] (S) to node {$t$} (B);

\draw[->, dashed] (Df) to node {$\gamma$} (Dg);
\draw[->, dashed] (Df) to node [swap] {$\br_2(f)$} (hX);
\draw[->, very thick, red, dashed] (A) to node {$\varphi_\fa$} (hX);
\draw[->, very thick, red, dashed] (S) to node {$\varphi_{s^*\fa}$} (hX);
\end{tikzpicture}
\caption{See proof of \THM{C}.}
\label{fig:r2b}
\end{figure}

We observe that $s^*[f]=t^*[g]$ as these
are the classes of the pullback of a hyperplane section of projectively isomorphic surfaces.
Since $[f]$ and $[g]$ are orthogonal to the classes
of $(-1)$-curves contracted by $s$ and $t$, \resp,
we observe that
$\gcd [f]=\gcd s^*[f]$ and $\gcd [g]=\gcd s^*[g]$
so that
\[
\gcd[f]=\gcd[g]
\quad\text{and}\quad
s^*\fa
=
\frac{1}{\gcd[f]}s^*[f]
=
\frac{1}{\gcd[g]}t^*[g]
=
t^*\fb.
\]
It follows that
$\hX=\img\varphi_{s^*\fa}$ is
projectively isomorphic to $\hY=\img\varphi_{t^*\fb}$.
Hence, the condition~$\bc_2$ is a projective invariant and
there exists a projective isomorphism
$\hat{\beta}\c\hX\to\hY$
that makes the diagram in \FIG{r2c} commutative.
We conclude that
$\gamma\in\cR(\br_2(f),\br_2(g))$
and thus $\br_2$ is a compatible reducer as asserted.

\begin{figure}[!ht]
\centering
\begin{tikzpicture}[node distance=13mm, auto]
\usetikzlibrary{arrows.meta}
\node (C1) {};
\node (C2) [right of=C1] {};
\node (C3) [right of=C2] {};
\node (S) [right of=C3]  {$S$};
\node (C5) [right of=S]  {};
\node (C6) [right of=C5] {};
\node (C7) [right of=C6] {};

\node (A) [below of=C1] {$\bmd f$};;
\node (B) [below of=C7] {$\bmd g$};

\node (X)  [below of=A] {$\img f$};
\node (E2) [right of=X]  {};
\node (hX) [right of=E2] {$\hX$};
\node (E4) [right of=hX]  {};
\node (hY) [right of=E4] {$\hY$};
\node (E6) [right of=hY]  {};
\node (Y)  [right of=E6] {$\img g$};
\node (Df) [below of=E2] {$\dom f$};
\node (Dg) [below of=E6] {$\dom g$};

\draw[->] (A) to node [swap] {$\varphi_{[f]}$} (X);
\draw[->] (B) to node {$\varphi_{[g]}$} (Y);
\draw[->, dashed] (Df) to node {$f$} (X);
\draw[->, dashed] (Dg) to node [swap] {$g$} (Y);

\draw[->] (S) to node [swap] {$s$} (A);
\draw[->] (S) to node {$t$} (B);

\draw[->, dashed] (Df) to node {$\gamma$} (Dg);

\draw[->, dashed] (Df) to node [swap] {$\br_2(f)$} (hX);
\draw[->, dashed] (Dg) to node {$\br_2(g)$} (hY);
\draw[->, dashed] [swap] (S) to node {$\varphi_{s^*\fa}$} (hX);
\draw[->, dashed] (S) to node {$\varphi_{t^*\fb}$} (hY);

\draw[->, very thick, red,] (hX) to node {$\hat{\beta}$} (hY);
\end{tikzpicture}
\caption{See proof of \THM{C}.}
\label{fig:r2c}
\end{figure}

We can now see that $\bp$ is a projective invariant as asserted.
Indeed, $[f]^2=\deg(\img f)$ and
$h^0([f])$ is
the embedding dimension of the image of $\br_0(f)$.
Since $\br_2$ is compatible it follows that $\gcd[f]$ is
a projective invariant as well.

It is only left to show that the reducer~$\br_1$ is compatible.
Thus we need to show that
$\bc_1$ is a projective invariant and
$\gamma\in\cR(\br_1(f),\br_1(g))$ if $(\bc_1(f),\bc_1(g))=(1,1)$.
We assume that \Wlog that $\bc_1(f)=1$.

As before, let $\alpha\c\bmd f\dto\bmd g$ be a birational map
\st
$\beta\circ \varphi_{[f]}=\varphi_{[g]}\circ\alpha$
and recall from the factorization theorem for birational maps
that there exists a smooth surface $S$ and birational morphisms
$s\c S\to \bmd f$ and $t\c S\to \bmd g$
so that the diagram in \FIG{r2a} commutes.
We set $\hX:=\img\br_1(f)$ and $\hY:=\img\br_2(g)$.
Let $\k$ denote the canonical class of $S$
and let $\fa$ denote the sum of the pullbacks of the classes of $(-1)$-curves
that are contracted by $s$.
We have $s^*\kf=\k-\fa$ and $s_*\k=\kf$ by \citep[Proposition~V.3.3]{har}.
We observe that
\[
\varphi_{s^*([f]+\kf)}=\varphi_{[f]+\kf}\circ s
\quad\text{and}\quad
s_*(s^*[f]+\k)=[f]+\kf,
\]
and thus $\varphi_{s^*[f]+\k}$ makes the diagram
of \FIG{r1b} commutative.
We remark that if $\fa\neq 0$,
then the linear series $|s^*[f]+\k|$
has a fixed part, since $(s^*[f]+\k)\cdot\fa<0$.

\begin{figure}[!ht]
\vspace{-5mm}
\centering
\begin{tikzpicture}[node distance=13mm, auto]
\usetikzlibrary{arrows.meta}
\node (C1) {};
\node (C2) [right of=C1] {};
\node (C3) [right of=C2] {};
\node (S) [right of=C3]  {$S$};
\node (C5) [right of=S]  {};
\node (C6) [right of=C5] {};
\node (C7) [right of=C6] {};

\node (A) [below of=C1] {$\bmd f$};;
\node (B) [below of=C7] {$\bmd g$};

\node (X)  [below of=A] {$\img f$};
\node (E2) [right of=X]  {};
\node (hX) [right of=E2] {$\hX$};
\node (E4) [right of=hX]  {};
\node (hY) [right of=E4] {};
\node (E6) [right of=hY]  {};
\node (Y)  [right of=E6] {$\img g$};
\node (Df) [below of=E2] {$\dom f$};
\node (Dg) [below of=E6] {$\dom g$};

\draw[->] (A) to node [swap] {$\varphi_{[f]}$} (X);
\draw[->] (B) to node {$\varphi_{[g]}$} (Y);
\draw[->, dashed] (Df) to node {$f$} (X);
\draw[->, dashed] (Dg) to node [swap] {$g$} (Y);

\draw[->] (S) to node [swap] {$s$} (A);
\draw[->] (S) to node {$t$} (B);

\draw[->, dashed] (Df) to node {$\gamma$} (Dg);
\draw[->, dashed] (Df) to node [swap] {$\br_1(f)$} (hX);
\draw[->, very thick, red, dashed] (A) to node {$\varphi_{[f]+\kf}$} (hX);
\draw[->, very thick, red, dashed] (S) to node {$\varphi_{s^*[f]+\k}$} (hX);
\end{tikzpicture}
\caption{See proof of \THM{C}.}
\vspace{-2mm}
\label{fig:r1b}
\end{figure}

Since $s^*[f]=t^*[g]$ by construction,
there exists a projective isomorphism $\hat{\beta}\in\cP(\br_1(f),\br_1(g))$
making the diagram of \FIG{r1c} commutative.

\begin{figure}[!ht]
\centering
\begin{tikzpicture}[node distance=13mm, auto]
\usetikzlibrary{arrows.meta}
\node (C1) {};
\node (C2) [right of=C1] {};
\node (C3) [right of=C2] {};
\node (S) [right of=C3]  {$S$};
\node (C5) [right of=S]  {};
\node (C6) [right of=C5] {};
\node (C7) [right of=C6] {};

\node (A) [below of=C1] {$\bmd f$};;
\node (B) [below of=C7] {$\bmd g$};

\node (X)  [below of=A] {$\img f$};
\node (E2) [right of=X]  {};
\node (hX) [right of=E2] {$\hX$};
\node (E4) [right of=hX]  {};
\node (hY) [right of=E4] {$\hY$};
\node (E6) [right of=hY]  {};
\node (Y)  [right of=E6] {$\img g$};
\node (Df) [below of=E2] {$\dom f$};
\node (Dg) [below of=E6] {$\dom g$};

\draw[->] (A) to node [swap] {$\varphi_{[f]}$} (X);
\draw[->] (B) to node {$\varphi_{[g]}$} (Y);
\draw[->, dashed] (Df) to node {$f$} (X);
\draw[->, dashed] (Dg) to node [swap] {$g$} (Y);

\draw[->] (S) to node [swap] {$s$} (A);
\draw[->] (S) to node {$t$} (B);

\draw[->, dashed] (Df) to node {$\gamma$} (Dg);

\draw[->, dashed] (Df) to node [swap] {$\br_1(f)$} (hX);
\draw[->, dashed] (Dg) to node {$\br_1(g)$} (hY);
\draw[->, dashed] [swap] (S) to node {$\varphi_{s^*[f]+\k}$} (hX);
\draw[->, dashed] (S) to node {$\varphi_{t^*[g]+\k}$} (hY);

\draw[->, very thick, red,] (hX) to node {$\hat{\beta}$} (hY);
\end{tikzpicture}
\caption{See proof of \THM{C}.}
\vspace{-2mm}
\label{fig:r1c}
\end{figure}

It follows that $\bc_1$ is a projective invariant and that $\gamma\in \cR(\br_1(f),\br_1(g))$.
We conclude that the reducer $\br_1$ is compatible as asserted.
\end{proof}

We will now proceed with the proof of \THM{B}.

In the remainder of this section we suppose that
$S$ is a smooth rational surface with canonical class $\k$.

The \df{nef threshold} of a class~$\fh\in N(S)$ is defined as
\[
\tau(\fh):=\operatorname{sup}\set{t\in\R}{\fh+t\,\k \text{ is nef}}.
\]
Notice that $\k$ is not nef as $S$ is a rational surface.
Hence $\fh$ is nef if and only if $\tau(\fh)\geq 0$.
Recall from \DEF{M} that the components of $f\in\cM$
do not have a non-constant greatest common divisor,
and thus $\tau([f])\geq 0$.

\begin{lemma}
\label{lem:t}
If $\fh\in N(S)$ \st $\tau(\fh)\geq 0$ and $\fh^2>0$,
then
$\tau(\fh)\in\Q_{\geq 0}$.
\end{lemma}

\begin{proof}
If $\fh$ is ample, then
the assertion follows from the rationality theorem at \citep[Theorem~1-2-11]{mat}.
The proof of \citep[Theorem~1-2-11]{mat} also works with $\fh$ nef and big
instead of ample by using the
Kawamata-Viehweg vanishing theorem
\citep[Theorem~4.3.1]{laz1} instead of the Kodaira Vanishing theorem.
We know from \citep[Theorem~2.2.16 (bigness of nef divisors)]{laz1} that $\fh$ is nef and big.
\end{proof}

\begin{lemma}
\label{lem:nef}
If $\fb,\fc\in N(S)$ \st $\tau(\fb),\tau(\fc)\geq 0$, then $\fb\cdot \fc\geq 0$.
\end{lemma}

\begin{proof}
See \citep[Example~1.4.16]{laz1}.
\end{proof}

\begin{lemma}
\label{lem:RR}
If $\fc\in N(S)$ \st $\tau(\fc)\geq 0$ and $\fc^2>0$,
then
\[
h^0(\fc+\k)=\tfrac{1}{2}\,\fc\cdot(\fc+\k)+1.
\]
\end{lemma}

\begin{proof}
We know from \citep[Theorem~2.2.16 (bigness of nef divisors)]{laz1} that $\fc$ is nef and big.
The assertion now follows from the Riemann-Roch theorem and
the Kawamata-Viehweg vanishing theorem
as stated at \citep[Theorem~4.3.1]{laz1}.
\end{proof}

\begin{lemma}
\label{lem:tinZ}
If $\fh\in N(S)$, $\tau(\fh)\geq 0$ and $(\fh+\tau(\fh)\,\k)^2\neq 0$,
then
$\tau(\fh)\in\Z_{\geq 0}$ and
there exists a $(-1)$-curve~$E\subset S$
\st $(\fh+\tau(\fh)\,\k)\cdot [E]=0$.
\end{lemma}

\begin{proof}
There exists a curve $E\subset S$ \st $(\fh+\tau(\fh)\,\k)\cdot [E]=0$ and $\k\cdot [E]<0$.
Indeed $E$ is a curve that determines the nef threshold.
Recall from \LEM{nef} that $(\fh+\tau(\fh)\,\k)^2>0$.
It follows from Hodge index theorem and the genus formula
that $[E]^2=\k\cdot[E]=-1$.
This concludes the proof as $\tau(\fh)=\fh\cdot [E]$.
\end{proof}

\begin{lemma}
\label{lem:tbound}
If $\fh\in N(S)$ and $\fc:=\alpha\,\fh+\beta\,\k$
for some co-prime $\alpha,\beta\in\Z_{>0}$ \st
\[
\tau(\fc-\k)\geq 0,\quad
(\fc-\k)^2>0
\quad\text{and}\quad
h^0(\fc)\leq 1,
\]
then either
$\tau(\fh)<{\beta}/{\alpha}$ or $\fc=0$.
\end{lemma}

\begin{proof}
It follows from \LEM{RR} that $\fc\cdot(\fc-k)\leq 0$.
If $\fc\cdot(\fc-k)=0$, then either $\fc^2<0$ or $\fc=0$ by the Hodge index theorem.
If $\fc^2<0$ or $\fc\cdot(\fc-k)<0$, then $\fc$ is not nef by \LEM{nef}
and thus $\tau(\fh)<{\beta}/{\alpha}$.
\end{proof}

\begin{lemma}
\label{lem:c1a}
If $\fh\in N(S)$ \st
\[
\tau(\fh)> 0,\quad
\fh^2>0
\quad\text{and}\quad
h^0(\fh+\k)\leq 1,
\]
then $(\fh+\tau(\fh)\,\k)^2=0$ and $\tau(\fh)\leq 1$.
\end{lemma}

\begin{proof}
It follows from \LEM{tbound} with $\alpha=\beta=1$ that either $\tau(\fh)<1$ or $\fh+\k=0$
so that $\tau(\fh)=1$.
Recall from \LEM{nef} that $(\fh+\tau(\fh)\,\k)^2\geq 0$.
If $(\fh+\tau(\fh)\,\k)^2>0$, then $\tau(\fh)\in\Z_{>0}$
by \LEM{tinZ} so that we arrive at a contradiction.
\end{proof}

\begin{lemma}
\label{lem:fiber}
Suppose that $M,F\in N(S)$ are the classes of the
moving and fixed part of the linear series $|M+F|$
\st
\[
h^0(M+F)>1
\quad\text{and}\quad
M^2=0.
\]
\begin{Mclaims}
\item[\bf a)]
If $M\cdot [A]=[A]^2=0$ for some curve $A\subset S$, then
$[A]=\beta\,M$
for some~$\beta\in\Q_{>0}$.

\item[\bf b)]
$M=\gamma\,[C]$ for some $\gamma\in\Z_{>0}$ and irreducible curve $C\subset S$.

\item[\bf c)]
If $\tau(M+F)\geq 0$, then $F=0$.

\end{Mclaims}
\end{lemma}

\begin{proof}
a) See \citep[Lemma~1-2-10 and Proposition~1-2-16]{mat}.

b)
Let $M=\sum_{i\in I} M_i$, where $M_i$ are classes of irreducible curves.
Since $M^2=0$ and $\tau(M)\geq 0$ we find that $M_i\cdot M_j=0$ for all $i,j\in I$
and thus this assertion is a consequence of a).

c)
Suppose by contradiction that $F\neq 0$.
We know from \LEM{nef} that $(M+F)^2\geq 0$.
Since $M^2=0$, $\img\varphi_{\alpha M}$ is a curve for all $\alpha\in\Z_{>0}$.
The class of the fixed part of $\alpha\,(M+F)=\alpha\,M+\alpha\, F$ is
$\alpha\, F$ and thus $\img\varphi_{\alpha M}\cong\img \varphi_{\alpha\,(M+F)}$.
Hence $\varphi_{\alpha\,(M+F)}$ is not birational so that
$(M+F)^2=0$ by \citep[Theorem~2.2.16 (bigness of nef divisors)]{laz1}.
Since $(M+F)^2=(M+F)\cdot F+M\cdot F+M^2=0$ and $(M+F)\cdot F,M\cdot F,M^2\geq 0$
it follows that $M\cdot F=F^2=0$. We arrived at a contradiction
as a) states that $F$ must be a multiple of~$M$ and thus cannot be a fixed part.
\end{proof}

\begin{lemma}
\label{lem:pencil}
If $\fh\in N(S)$ \st
\[
\tau(\fh)>0,\quad
\fh^2>0,\quad
\bigl(\fh+\tau(\fh)\,\k\bigr)^2=0
\quad\text{and}\quad
\fh+\tau(\fh)\,\k\neq 0,
\]
then
$\fh+\tau(\fh)\,\k=\gamma\,[C]$ for some $\gamma\in\frac{1}{2}\Z_{>0}$
and a rational curve $C\subset S$
\st
\[
\tau(\fh)\in\tfrac{1}{2}\Z_{> 0},\quad
[C]^2=0,\quad
\k\cdot [C]=-2
\quad\text{and}\quad
h^0([C])>1.
\]
\end{lemma}

\begin{proof}
Recall from \LEM{t} that $\tau(\fh)={\beta}/{\alpha}$ for some co-prime $\alpha,\beta\in\Z_{>0}$.
Thus $\tau(\alpha\,\fh+\beta\,\k)\geq 0$ and $\tau(\alpha\,\fh+(\beta-1)\,\k)\geq 0$
so that $h^0(\alpha\,\fh+\beta\,\k)\geq 1$ by the Riemann-Roch theorem and \LEM{nef}.
It follows from \LEM{fiber}b that $\alpha\,\fh+\beta\,\k=\gamma'\,[C]$ for
some irreducible curve $C\subset S$ and $\gamma'\in\Z_{>0}$.
If $\fh\cdot [C]=0$, then we arrive at a contradiction with Hodge index theorem
and thus $\fh\cdot [C]>0$.
Notice that $\fh\cdot [C]+\tau(\fh)\,\k\cdot [C]=0$ and thus
$\k\cdot [C]<0$ so that $\k\cdot [C]=-2$ by the genus formula.
Therefore $\tau(\fh)=\frac{1}{2}\,\fh\cdot [C]$
which concludes the proof.
\end{proof}

\begin{lemma}
\label{lem:adjoint}
Suppose that $\fh\in N(S)$ \st
$\tau(\fh)\geq 0$ and $\fh^2>0$.
\begin{Mclaims}
\item[\bf a)]
If $\tau(\fh)=0$, then
there exists a birational morphism $\nu\c S\to S'$ to
a smooth surface~$S'$
\st
\begin{gather*}
\tau(\nu_*\fh)>0,\quad
\img\varphi_{\nu_*\fh}=\img\varphi_\fh,\quad
h^0(\nu_*\fh)=h^0(\fh),\quad
\gcd\nu_*\fh=\gcd\fh,\quad
\\
(\nu_*\fh)^2=\fh^2
\quad\text{and}\quad
\nu_*\fh\cdot\nu_*\k=\fh\cdot\k.
\end{gather*}

\item[\bf b)]
If $\fc:=\fh+\k$ and $h^0(\fc)>1$, then
there exists a birational morphism $\mu$
\st
$\dom\mu=S$, $\img\mu$ is a smooth surface and
\[
\tau(\mu_*\fc)\geq 0
\quad\text{and}\quad
\img\varphi_{\mu_*\fc}=\img\varphi_\fc.
\]
\end{Mclaims}
\end{lemma}

\begin{proof}
a)
It follows from \LEM{tinZ} that there exists
a $(-1)$-curve $E\subset S$ \st $\fh\cdot [E]=0$.
By Castelnuovo's contraction theorem there exists a birational
morphism $\nu_1\c S\to S_1$ that contracts $E$ to a smooth point.
If $\tau(\nu_{1*}\fh)>0$, then $\nu:=\nu_1$
and the remaining assertions are a straightforward consequence
of basic intersection theory (see for example \citep[Section~V.3]{har}).
If $\tau(\nu_{1*}\fh)=0$, then we repeat the same argument
for $\nu_{1*}\fh$ and contract the resulting $(-1)$-curve.
Since $\rnk N(S_1)<\rnk N(S)<\infty$ there will be a finite number
of contractions
and
we set $\nu$ equal to the composition of these contractions.

b)
Suppose that $E\subset S$ is an irreducible curve \st $\fc\cdot [E]<0$.
This implies that $h^0([E])=1$.
We have $\fh\cdot [E]\geq 0$ and thus $\k\cdot [E]<0$.
By the Riemann-Roch formula and Serre duality we have $[E]^2-[E]\cdot\k\leq 0$ and thus $[E]^2<0$.
Therefore $E$ is a $(-1)$-curve by the genus formula.
We now apply Castelnuovo's contraction theorem as in a)
and define $\mu$ as the compositions of
contractions of $(-1)$-curves that are negative against pushforwards of $\fc$.
\end{proof}

We call $(S,\fh)$ a \df{reduction pair} for $f\in\cM$
if the following four properties are satisfied:
\begin{enumerate}[topsep=0pt]
\item $\tau(\fh)>0$ and $\fh^2>0$,
\item $\img \varphi_\fh=\img \Psi_{[f]}$ (see \DEF{map2}),
\item $\bigl(h^0(\fh),~\fh^2,~\gcd\fh\bigr)=\bigl(h^0([f]),~[f]^2,~\gcd[f]\bigr)$,
\item $h^0(\fh+\k)=h^0([f]+\kf)$.
\end{enumerate}

\begin{lemma}
\label{lem:pair}
There exists a reduction pair $(S,\fh)$ for all $f\in\cM$ \st $[f]^2>0$.
\end{lemma}

\begin{proof}
Let $(S,\fh):=(\bmd f,[f])$.
If $\tau(\fh)>0$, then $(S,\fh)$ is a reduction pair for~$f$
as a direct consequence of the definitions.
Now suppose that $\tau(\fh)=0$
and
let $\nu\c S\to S'$ be defined as in \LEM{adjoint}a.
Since $\nu_*\k$ is the canonical class of~$S'$,
it follows from \LEM{RR} that $h^0(\nu_*\fh+\nu_*\k)=h^0(\fh+\k)$
so that $(S',\nu_*\fh)$ is a reduction pair for~$f$.
\end{proof}

\begin{lemma}
\label{lem:B45}
If $(S,\fh)$ is a reduction pair for $f\in\cM$
\st
\[
(\fh+\tau(\fh)\,\k)^2=0,\quad
\tau(\fh)\leq 1
\quad\text{and}\quad
\fh+\tau(\fh)\,\k\neq 0,
\]
then $f$ is characterized by either base case~B4 or B5.
\end{lemma}

\begin{proof}
We know from \LEM{pencil}
that $\tau(\fh)\in\{\frac{1}{2},1\}$ and $\fh+\tau(\fh)\,\k=\gamma\,[C]$
for some rational curve $C\subset S$ \st $h^0([C])>1$ and $\k\cdot [C]=-2$.
It follows that $1\leq \fh\cdot [C]\leq 2$.
If $\fh\cdot [C]=1$, then $\varphi_\fh(S)$ is
covered by lines so that $f$ is characterized by base case~B4.
If $\fh\cdot [C]=2$, then $C$ is mapped by $\varphi_\fh$ either 2:1 to a line or 1:1 to a conic.
Hence $f$ is in this case characterized by base case~B5.
\end{proof}

\begin{lemma}
\label{lem:B123}
If $(S,\fh)$ is a reduction pair for $f\in\cM$
\st
\[
\fh+\tau(\fh)\,\k=0 \quad\text{and}\quad
\gcd\fh=1,
\]
then $f$ is characterized by base case either
B1, B2 or B3.
\end{lemma}

\begin{proof}
By assumption $\fh=-\tau(\fh)\,\k$ and thus $-\k$ is nef
so that $S$ is a weak del Pezzo surface \cite[Definition~8.1.18]{dp}.
It follows from \cite[Theorem~8.1.15, Lemma~8.3.1, Theorem~8.3.2]{dp}
that $S$ is a weak del Pezzo surface of degree~$1\leq \fh^2\leq 8$ \st $h^0(-\k)=\k^2+1$.
Since $\gcd\fh=1$ we find that $\tau(\fh)\in\{\frac{1}{3},\frac{1}{2},1\}$ by
\citep[Corollary 1-2-15 (boundedness of denominator)]{mat}.
The assertions about birationality are a consequence of Reider's theorem.
Therefore $f$ is characterized by base case B1, B2 or B3
if $\tau(\fh)$ is equal to $\frac{1}{3}$, $\frac{1}{2}$ and $1$, \resp.
\end{proof}

\begin{lemma}
\label{lem:B5direct}
If $f\in\cM$ and $c:=[f]+\kf$ \st
\[
h^0(c)>1
\quad\text{and}\quad
[f]^2>\mc{c}^2=c\cdot\mc{c}=0,
\]
then $f$ is characterized by base case~B5.
\end{lemma}

\begin{proof}
We know from \LEM{fiber} that $\mc{c}=\gamma\,[C]$
for some irreducible curve $C\subset S$ and $\gamma\in\Z_{>0}$.
We have $[f]\cdot [C]>0$, otherwise we arrive at a contradiction
with the Hodge index theorem.
Since $([f]+\kf)\cdot [C]=0$ it follows that $\kf\cdot [C]<[C]^2=0$.
Hence $\kf\cdot [C]=-2$ by the genus formula so that  $C$ is a rational curve.
The image of $f$ is a surface, since $[f]^2>0$ and $[f]\cdot [C]=2$.
Therefore $C$ is mapped by $\varphi_{[f]}$ to either a line or a conic.
We conclude that $f$ is characterized by base case~B5.
\end{proof}

\begin{lemma}
\label{lem:zero}
If $f\in \cM$ \st $[f]^2=\bc_0(f)=\bc_1(f)=\bc_2(f)=0$, then
there exists no $g\in\cM_{\bc_1}$ \st
$[g]^2>0$ and
$f\in\{\br_1(g),~ (\br_2\circ\br_1)(g)\}$.
\end{lemma}

\begin{proof}
Suppose by contradiction that there exists $g\in\cM_{\bc_1}$ \st
$[g]^2>0$ and $f=\br_1(g)$.
Let $c:=[g]+\kg$ and notice that $\mc{c}=[\br_1(g)]=[f]$
is the class of the moving part of the linear series $|c|$
associated to~$c$.
Since $\bc_1(g)=1$, $[g]^2>0$ and $\mc{c}^2=0$,
we deduce from the definition of~$\bc_1$ that $h^0(c)>1$ and $\mc{c}\cdot c\neq 0$.
Thus $c-\mc{c}\neq 0$ and there exists an irreducible curve $A\subset\bmd g$
\st $[A]$ is the class of a prime divisor of the fixed part with class $c-\mc{c}$.
By \LEM{adjoint}b there exists a birational morphism
$\mu\c\bmd g\to\img\mu$ \st
$\tau(\mu_*c)\geq 0$
and
$\img\varphi_{\mu_*c}=\img\varphi_c$.
It follows from \LEM{fiber}c that $(\mu_*c)^2=0$ and $\mu_*[A]=0$.
Hence $A$ is contained the fiber of~$\varphi_{\mc{c}}$.
This implies that $\mc{c}\cdot A=0$ and thus we deduce that $\mc{c}\cdot(c-\mc{c})=0$.
We arrived at a contradiction as $\mc{c}\cdot(c-\mc{c})=\mc{c}\cdot c=0$.

Finally, suppose by contradiction that there exists $g\in\cM_{\bc_1}$ \st
$[g]^2>0$ and $f=(\br_2\circ\br_1)(g)$.
In this case
$[h]^2=\bc_0(h)=\bc_1(h)=0$, where $h:=\br_1(g)$.
We now arrive at a contradiction by applying the above arguments
with $h$ instead of $f$.
\end{proof}

\begin{proof}[Proof of \THM{B}.]
a)
If $h^0([f]+\kf)>1$, then
it follows from \LEM{B5direct} that $f$ is characterized by base case~B5.
Now suppose that $h^0([f]+\kf)\leq 1$.
By \LEM{pair} there exists a reduction pair~$(S,\fh)$ for $f$.
It follows from \LEM{c1a} that $(\fh+\tau(\fh)\,\k)^2=0$ and $\tau(\fh)\leq 1$.
If $\fh+\tau(\fh)\,\k\neq 0$, then
it follows from \LEM{B45} that $f$ is characterized by base case~B4 or B5.
If $\fh+\tau(\fh)\,\k=0$, then we know from \LEM{B123}
that $f$ is characterized by base case~B1, B2 or B3.

b) This assertion follows from \LEM{zero}.
\end{proof}

\section{Acknowledgements}

We would like to thank J.G. Alc{\'a}zar for
bringing interesting ideas to our attention.
Our software implementation \cite{github} uses \citep[Sage]{sage}.
This research was supported by the Austrian Science Fund (FWF): project P33003.

\section{References}
\bibliography{surface-iso}

\paragraph{}
B. J\"uttler,
Institute of Applied Geometry,
Johannes Kepler University
\\
\textbf{email:} \url{bert.juettler@jku.at}

\paragraph{}
N. Lubbes,
Johann Radon Institute for Computational and Applied
Mathematics (RICAM), Austrian Academy of Sciences
\\
\textbf{email:} \url{niels.lubbes@gmail.com}

\paragraph{}
J. Schicho,
Research Institute for Symbolic Computation (RISC),
Johannes Kepler University
\\
\textbf{email:} \url{josef.schicho@risc.jku.at}
\end{document}